\documentclass[10pt]{article}

\usepackage{amssymb,amsthm,amsmath,amsfonts}
\usepackage{graphicx}
\usepackage{subfigure}

\usepackage[usenames]{color}
\graphicspath{{./pics/}}
\usepackage[colorlinks,linktocpage,linkcolor=blue]{hyperref}
\usepackage{algorithm}
\usepackage{algorithmicx}
\usepackage{algpseudocode}

\newtheorem{thm}{Theorem}[section]

\newtheorem{lem}[thm]{Lemma}

\newtheorem{exam}{Example}

\newtheorem{algo}[thm]{WG algorithm}



\numberwithin{equation}{section}

  \def\nb{\nonumber}
\def \Vh0{\stackrel{\circ}{V}_h} \def\to{\rightarrow}

\def\ra{\rangle} \def\la{\langle}

\newcommand{\lc}
{\mathrel{\raise2pt\hbox{${\mathop\la \limits_{\raise1pt\hbox
{\mbox{$\sim$}}}}$}}}

\newcommand{\gc}
{\mathrel{\raise2pt\hbox{${\mathop\ra\limits_{\raise1pt\hbox{\mbox{$\sim$}}}}$}}}

\newcommand{\ec}
{\mathrel{\raise2pt\hbox{${\mathop=\limits_{\raise1pt\hbox{\mbox{$\sim$}}}}$}}}

\def\beqn{\begin{eqnarray}}  \def\eqn{\end{eqnarray}}

\def\beqnx{\begin{eqnarray*}} \def\eqnx{\end{eqnarray*}}

\def\bn{\begin{enumerate}} \def\en{\end{enumerate}}

\def\bd{\begin{description}} \def\ed{\end{description}}

\title{{Weak Galerkin Method for Electrical Impedance Tomography}}
\author{
Ying Liang\footnote{Department of Mathematics, Purdue University, West Lafayette, IN 47907. (liang402@purdue.edu).}
\and Jun Zou\footnote{Department of Mathematics, The Chinese University of Hong Kong, Shatin, N.T., Hong Kong.The work of this author was substantially
supported by Hong Kong RGC General Research Fund (Project 14306718) and NSFC/Hong Kong RGC Joint Research Scheme 2016/17
(Project N\_CUHK437/16).
(zou@math.cuhk.edu.hk).}
}

\begin{document}
\date{}
\maketitle

\begin{abstract}
We propose and analyze a weak Galerkin method for electrical impedance tomography based on the bounded variation regularization.   We use the complete electrode model to describe the forward process,  approximate the system by a weak Galerkin formulation with the lowest order, and demonstrate the advantages of the method for recovering the piecewise constant conductivities. 
The error estimate of the weak Galerkin method for the direct problem is derived and then utilized to establish the convergence of the proposed algorithm for the inverse problem. 
Numerical examples are presented to verify the effectiveness and efficiency of the weak Galerkin method.
\end{abstract}

\noindent {\footnotesize {\bf Mathematics Subject Classification(MSC2000)}: 35R30, 65F22, 65N21, 65N30, 86A22}

\section{Introduction}\label{sec1}

Electrical impedance tomography (EIT) is a technique of estimating the unknown internal physical conductivity of an object from the noisy voltage measurements on the surface, which has aroused considerable interest in nondestructive testing \cite{hallaji2014electrical}, breast
cancer detection \cite{borsic2010vivo}, stroke classification \cite{malone2014stroke}, and geophysical prospecting \cite{yu2000modified},  etc.
A typical experimental setup is considered in this work.
First, currents are driven into the concerned object by a set of electrodes attached to the surface. Then the
induced electric voltages will be measured on some electrodes to reconstruct the interior electrical properties.
This procedure is usually conducted several times with different input currents to collect sufficient information for the estimate of the interior conductivity. This electromagnetic process can be mathematically described by the complete electrode model \cite{cheng1989electrode, somersalo1992existence} in many applications.

Due to its growing popularity, many algorithms have been proposed for the EIT inverse problem (see e.g., \cite{lechleiter2008factorization, scherzer2011handbook, jin2012reconstruction, jin2012analysis, lechleiter2006newton, winkler2014resolution, knudsen2009regularized}). 
EIT is known to be highly ill-posed and nonlinear. In particular, the reconstruction is often very unstable with respect to 
the boundary measurements. Therefore, the development of some accurate and flexible forward solvers and the use of proper regularization techniques in the inverse process are crucial to handle the noisy measurements and the complex geometric boundary associated with the complete electrode model effectively.  
We will develop and analyze a novel weak Galerkin (WG) method as the forward solver, and 
follow a variational approach in the inverse process to seek the reconstruction of the physical conductivity 
by minimizing an appropriate functional 
with the bounded variation (BV) regularization.

Various numerical methods have been applied to solve the  direct problem in EIT, 
including 
 the boundary element method \cite{daneshmand20133d}, the mesh-free method \cite{yousefi2013combined}, the adaptive finite element method \cite{jin2016convergent}, and the stochastic Galerkin finite element method \cite{hakula2014reconstruction}.
We will develop a WG method as the forward solver and analyze its convergence.
The WG method is a class of numerical schemes that allows numerical approximations to be totally discontinuous \cite{wang2013weak}.
The main idea of the WG method is to interpret partial differential operators as weak differential operators, i.e., the generalized distributions over the space of discontinuous functions that include boundary information.
Then the partial differential equations can be reformulated as 
variational equations with the help of discrete weak differential operators. 
This WG method has been developed for a number of PDEs, 
e.g., bi-harmonic problems \cite{mu2014weak}, Stokes flow \cite{wang2016weak},  Signorini and obstacle problems \cite{zeng2017convergence}  for an incomplete list.

The WG method has several important strengths:
(1) the mesh partition can be of polytope type, i.e., any type of polygons in 2D or polyhedra in 3D; 
(2) the weak finite element space is easy to construct with general stability and approximation requirements;
and (3) the WG schemes can be hybridized so that some unknowns associated with the interior of each element can be locally eliminated, yielding a system of linear equations involving much fewer unknowns than what it appears. 
More importantly, as the WG method allows the separation of degrees of freedom on the edges (or faces) and interiors of 
each finite element, it can more accurately recover the boundaries of inhomogeneous inclusions as well as 
approximate the jump discontinuities in the current density at the edge of electrodes of 
non-zeros contact impedance on the boundary due to the complex physical setup of the complete electrode model. 


Moreover, when implementing the WG method, one
does not need to appropriately choose any (large) parameters for stabilizers as in most discontinuous Galerkin (DG) methods. In addition, the WG method is locally conservative by its design, and the normal fluxes across the interfaces of elements are continuous while there is no continuity in the DG fluxes \cite{Bastian2003dg}.
In addition, while special preconditioners or fast solvers are typically required to solve the resulting indefinite discrete linear systems (saddle-point problems) arising from mixed finite element method for elliptic problems \cite{brezzi1991mixed}, the WG method will always result in symmetric definite linear systems \cite{lin2015comparative}. These advantages make the WG method well-suited for approximating the EIT problem.

Inspired by \cite{mu2015weak}, we propose in this work a WG scheme that uses a new combination of polynomial spaces
different from the original WG scheme in \cite{wang2013weak}.  Specifically, instead of using polynomials of degree $k$ both in the interior and on the edge (or face) of each element, we use polynomials of degree $k-1$ on the edge (or face) and polynomials of degree $k$ in the interior of elements. In this way, the scheme reduces the number of unknowns without compromising its accuracy when compared with the initial scheme in \cite{wang2013weak}. In our analysis and numerical implementations, the proposed WG scheme with the lowest order polynomial approximation, i.e., piecewise linear polynomials in the interior of each element and constant on the edge (or face) of each element, is used. This reduces by half the size of unknowns compared with the original scheme in \cite{wang2013weak} with piecewise linear approximation in the interior and boundary of each element,  and the convergence order for the forward problem remains to be optimal.

For the inverse process, we will apply reconstruction algorithms based on a bounded variation (BV) regularization. 
The conventional choice of the regularization is $L^2$ norm or $H^1$ seminorm. 
The former was studied in \cite{Lukaschewitsch2003Tik} for the EIT problem, but the performance is not satisfactory as there exist undesired oscillations in the reconstruction. The latter, as analyzed in \cite{jin2012analysis}, may 
significantly smoothen the discontinuity in the solution as well as the boundaries of inhomogeneous inclusions.
In practice, however, the recovering conductivity is expected to be only piecewise regular and have sharp gradients in many applications such as geophysical imaging \cite{yu2000modified}, non-destructive testing \cite{hallaji2014electrical}.
In contrast, BV regularization has been proved to be successful 
in retaining sharp features of sought-for parameters and applied to 
the EIT inverse problem in the level-set method setting \cite{Chung2005tv}.
We further carry out the convergence analysis for the proposed BV-based reconstruction algorithm. 


Unlike the $L^2$ or $H^1$ regularizations, the BV regularization is non-differentiable. 
As a result, typical gradient-based methods such as the gradient 
descent method and Newton's method are not applicable for optimization. 
To numerically solve the minimization problem based on the BV regularization, we apply the fast iterative shrinkage/thresholding algorithm  (FISTA) \cite{beck2009fast}. One of its characteristics is that 
a combination of two previous iterates is used at each iteration to enjoy a better global rate of convergence than other existing projection-based methods. 
We remark that an auxiliary dual problem is also solved in each iteration 
using the WG method to calculate the G\^{a}teaux derivative 
of the differentiable part, i.e., the discrepancy functional, 
and the gradient update is followed by a proximal step 
associated with the BV regularization.

This paper is organized as follows.  In Section~\ref{sec2}, we present the complete electrode model for the forward process and some necessary notations. Then we describe the EIT inverse problem with BV regularization and 
prove the existence and stability of minimizers in Section~\ref{sec3}. 
Section~\ref{sec4} is dedicated to the description of the WG formulation and the error analysis. The convergence analysis of our algorithm for the discrete EIT inverse problem is provided in Section~\ref{sec5}.
In Section~\ref{sec6}, we present the results of numerical experiments to demonstrate the effectiveness of the proposed algorithm.
Finally, we conclude this work in Section~\ref{conclusion}, and 
the numerical algorithm of solving discrete EIT inverse problem by FISTA is presented in Appendix~\ref{app1}.  
Throughout the paper, we use the standard notation for Sobolev spaces as in \cite{Evans2010pde} and 
adopt $(\cdot ,\cdot)$ and $\langle \cdot ,\cdot\rangle$ for 
the inner products in $[L^2(\Omega)]^d$ and Euclidean space respectively. The notation $C$ denotes a generic constant that may differ in each condition, but it is always independent of the mesh size in the numerical discretization and other quantities of interest.

\section{Preliminaries}\label{sec2}

In this section, we begin by revisiting the complete electrode model (CEM), an accurate forward model for the EIT problem. Next, we present the weak formulation of the CEM and introduce some continuity results of the forward operator, which will be essential for the subsequent analysis.

As a standard model for medical applications of EIT, CEM can capture some practical features of the EIT problem, such as the discrete nature of electrodes and contact impedance effect; see \cite{cheng1989electrode, somersalo1992existence} for details.
Let $\Omega$ be an open bounded domain in $\mathbb{R}^d$ ($d=2,3$)  with a polyhedral boundary denoted by $\Gamma$.
Consider a set of electrodes, denoted by $\{e_l\}_{l=1}^L$, 
which are line segments/planar surfaces on $\Gamma$ and disjoint from each other, i.e., 
$\bar{e}_i\cap\bar{e}_k=\phi$ if $i\neq k$. 
The applied current on the $l$-th electrode $e_l$ is denoted as $I_l$, and we define the current vector $I=(I_1, \ldots, I_L)^t\in\mathbb{R}_\diamond^L$, where $\mathbb{R}_\diamond^L:=\{I\in \mathbb{R}^L: ~\sum_{l=1}^L I_l=0\}$, in accordance with the law of charge conservation.
The electrode voltage is denoted by $U=(U_1,\ldots, U_L)^t$ and is normalized such that $U\in \mathbb{R}_\diamond^L$. With these notations, we can present the mathematical formulation of the CEM as follows:
given the conductivity $\sigma$,  positive contact impedances $\{z_l\}_{l=1}^L$, and the input current $I\in \mathbb{R}_\diamond^L$, find the potential $u\in H^1(\Omega)$ and the electrode voltage $U\in \mathbb{R}_\diamond^L$ such that
\begin{equation}\label{contieqn}
\left\{
\begin{array}{rlll}
  -\nabla\cdot (\sigma\nabla u) &=& 0 & \mbox{ in } \Omega,\\
  u+z_l\sigma\frac{\partial u}{\partial n} &=& U_l &\mbox{ on } e_l \mbox{ for } l=1, \ldots,L,\\
 \int_{e_l}\sigma \frac{\partial u}{\partial n}\,ds&=& I_l& \mbox{ for }l=1, \ldots, L, \\
  \sigma \frac{\partial u}{\partial n} &=& 0& \mbox{ on } \Gamma\backslash \cup_{l=1}^L e_l.
  \end{array}
\right.
\end{equation}
%
The governing equation in \eqref{contieqn} is derived under a quasi-static assumption on the electromagnetic field in the interior of the concerned object. The second equation reflects the contact impedance effect: when electrical currents are injected into the object, a highly resistive thin layer with surface impedance $z_1, \cdots, z_L$ 
forms at the electrode-electrolyte
interface, leading to potential drops across the electrode-electrolyte interface. The value of potential
drop is determined by the product of the surface impedance $z_l$ and the current density $\sigma\frac{\partial u}{\partial n}$  by Ohm's law. As 
metallic electrodes are perfect conductors, the voltage $U_l$ is constant on each electrode. The
third equation describes that when injected through the electrode $e_l$, the current $I_l$ is completely confined within each electrode itself.

Considering the fact that physical conductivities are often discontinuous and have (possibly large) jumps 
across the interfaces between homogeneous and inhomogeneous media, 
we introduce the admissible set $\mathcal{A}$ that allows these physical features. 
We first define the total variation \cite{asas1999regularization} of a function $q \in L^1(\Omega)$ by
\begin{equation}\label{defbv}
\int_\Omega |Dq|=\sup\, \left\{\int_\Omega q \ \text{div} \ \mathbf{g} \, dx :  \mathbf{g}\in[ C_0^\infty(\Omega)]^d \mbox{ and } |\mathbf{g}(x)|_\infty \leq 1, \  x\in\Omega\right\},
\end{equation}
where  $C_0^\infty(\Omega)$ denotes the space of infinitely differentiable functions with compact support in $\Omega$, and $|\cdot|_\infty$ denotes the $l_\infty$ norm of vectors in $\mathbb{R}^d$, that is, for $x=(x_1,\cdots,x_d)^t$,
$$|x|_\infty: = \max_{1\leq i\leq d} |x_i|.$$ 
Then we can define the space of functions with bounded variation:
\beqn\nb
BV(\Omega) = \Big\{q\in L^1(\Omega): \int_\Omega \vert D q \vert < \infty \Big\},
\eqn
which is a Banach space endowed with the norm 
$
\Vert q\Vert_{BV(\Omega)}:=\Vert q\Vert_{L^1(\Omega)}+\int_\Omega |Dq|.
$
As physical conductivities $\sigma$ are naturally bounded both from
below and above, we come to define the admissible set $\mathcal{A}$ as 
\beqn\label{def:A}
\mathcal{A}=\{\sigma\in BV(\Omega): \lambda\leq\sigma(x)\leq \lambda^{-1} \text{ a.e. }  x\in \Omega\}\, 
\eqn
for some constant $\lambda\in (0,1)$. 
For the subsequent analysis, we introduce the lower-semicontinuity of the total variation below.  The proof can be found in \cite{giusti1984minimal}.
\begin{lem}
\label{lsc_tv}
Let $q \in BV(\Omega)$ and $\{ q_n \}_{n=1}^\infty \subset BV(\Omega)$. 
Suppose $q_n \to q$ in $L^1(\Omega)$ as $n\to \infty$. Then 
\beqn\nb
\int_\Omega \vert Dq \vert \leq \liminf\limits_{n\to\infty} \int_\Omega \vert D q_n \vert.
\eqn
Moreover, the embedding of $BV(\Omega)$ in $L^1(\Omega)$ is compact.
\end{lem}

Before presenting the weak formulation of the CEM, we introduce the product space $\mathbb{H}: = H^1(\Omega)\otimes\mathbb{R}_\diamond^L$ with its norm defined by
$$\Vert (u, U)\Vert_{\mathbb{H}}^2=\Vert u\Vert_{H^1(\Omega)}^2+\Vert U\Vert^2.$$ 
Then the weak formulation of the model \eqref{contieqn} reads: find $(u,U)\in \mathbb{H}$ such that
\begin{equation}\label{weakfor}
a(\sigma, (u, U), (v, V))=\langle I, V\rangle \quad \forall(v,V)\in \mathbb{H},
\end{equation}
where the map
 $a(\sigma, (u,U), (v, V)): \mathcal{A}\times\mathbb{H}\times\mathbb{H} \to \mathbb{R}$ is defined by
\beqn\nonumber
a(\sigma, (u, U), (v, V))=(\sigma\nabla u,\nabla v)+\sum_{l=1}^Lz_l^{-1}\langle u-U_l, v-V_l\rangle_{e_l},
\eqn
with $\langle \cdot ,\cdot \rangle_{e_l}$ being the inner product on $L^2(e_l)$.  
We have the following norm equivalence on the space $\mathbb{H}$ \cite[Lemma 2.1]{jin2016convergent}.
\begin{lem}\label{norm}
On the space $\mathbb{H}$, the norm $\Vert \cdot\Vert_{\mathbb{H}}$ is equivalent to the norm $\Vert\cdot\Vert_{\mathbb{H},\star}$ defined  by
\beqn\nb
\Vert (u, U)\Vert_{\mathbb{H},\star}^2=\Vert\nabla u\Vert_{L^2(\Omega)}^2+\sum_{l=1}^L\Vert u-U_l\Vert_{L^2(e_l)}^2.
\eqn
\end{lem}

It follows Lemma \ref{norm} that $a(\sigma, \cdot,\cdot)$ is continuous and coercive in $\mathbb{H}$ 
for given $\sigma\in \mathcal{A}$, hence we have the existence and uniqueness of the solution $(u, U)$  
to \eqref{weakfor} \cite{Somersalo1992eit} by the Lax-Milgram theorem.
We further introduce the notation  $\mathcal{F}(\sigma)$ to denote the forward operator, i.e., $(u, U)=(u(\sigma), U(\sigma))=\mathcal{F}(\sigma)\in \mathbb{H}$. 
For simplicity, we suppress the dependence of the solution $(u,U)$ on the input current $I$.
We end this section with several continuity results of the forward operator $\mathcal{F}(\sigma)$ \cite[Lemma 3.1, Lemma 3.2 ]{jin2012analysis}.
\begin{lem}
\label{lemma2}
The operator $\mathcal{F}(\sigma):\mathcal{A}\to \mathbb{H}$ is uniformly bounded for a fixed $I$.
\end{lem}
\begin{lem}
\label{staF}
For the operator $\mathcal{F}(\sigma)$ and $\sigma$, $\sigma+\vartheta\in \mathcal{A}$, 
we have the following continuity estimates:
\begin{enumerate}
\item For any $p\in \left(\frac{2Q(\lambda)}{Q(\lambda)-2},\infty\right]$, 
$$\Vert \mathcal{F}(\sigma+\vartheta)-\mathcal{F}(\sigma)\Vert_{\mathbb{H}}\leq C\Vert \vartheta\Vert_{L^p(\Omega)};$$
\item For any $p\in \left(\frac{4Q(\lambda)}{Q(\lambda)-2},\infty\right]$, there exists a $q\in (2, Q(\lambda))$ such that
\beqn\nb
\Vert u(\sigma+\vartheta)-u(\sigma)\Vert_{W^{1,q}(\Omega)}\leq C\Vert \vartheta\Vert_{L^p(\Omega)};
\eqn

\item For $p\geq 1$ and any $q\in (2, Q(\lambda))$, 
\beqn\nb
\lim_{\Vert\vartheta\Vert_{L^p(\Omega)}\to 0}\Vert  u(\sigma+\vartheta)-u(\sigma)\Vert_{W^{1,q}(\Omega)}=0,
\eqn
\end{enumerate}
where $Q(\lambda)>2$ is a constant  depending on $d$ and $\lambda$ only and tending to $\infty$ and $2$ as $\lambda\to 1$ and $\lambda\to 0$ respectively.
\end{lem}

\section{{Bounded variation regularization}}\label{sec3}

In this section, we will formulate the minimization problem for the EIT  inverse problem based on BV regularization and investigate some of its analytic properties.

Recall that the EIT inverse problem is to reconstruct the conductivity $\sigma$ from noisy measurements $U^\delta$ of the exact electrode voltage $U(\sigma^\dag)$, corresponding to one or multiple input currents. 
For the analysis, we will use the data generated by one set of input currents.
As EIT is severely ill-posed,  some appropriately chosen regularization is crucial 
to combat the numerical instability and generate physically meaningful images. 
For this purpose, we use the BV regularization and 
the following least-squares approach to reconstruct $\sigma^\dag$: 
%
\begin{equation}\label{mini}
\min_{\sigma\in\mathcal{A}}\left\{J(\sigma)=\frac{1}{2}\Vert U(\sigma)-U^\delta\Vert^2+\alpha N(\sigma)\right\},
\end{equation}
where $N(\sigma)$ is the BV-penalty term, 
i.e., 
$ N(\sigma)=\int_\Omega |D\sigma|\, $, 
and $\alpha$ is a scalar compromising the discrepancy term and penalty term. 
The following theorem states that the minimization problem \eqref{mini} has at least one solution.  
\begin{thm}\label{exist}
There exists at least one minimizer to problem \eqref{mini}.
\end{thm}
\begin{proof}
Lemma~\ref{lemma2} implies that
$\inf J(\sigma)$ is finite over $\mathcal{A}$. Then there exists a minimizing sequence $\{\sigma_n\}_{n=1}^\infty \subset \mathcal{A}$ such that
\beqn\nb
\lim_{n\to\infty}J(\sigma_n)=\inf_{\sigma\in\mathcal{A}} J(\sigma) .
\eqn
To simplify the notations, we denote $ u(\sigma_n)$ by $u_n$ and $U(\sigma_n)$ by $U_n  =(U_{n,1}, U_{n,2},...,U_{n,L})^t$.
It follows from the definition of $J$ that the sequences $\{\Vert U_n-U^\delta\Vert\}_{n=1}^\infty$ and 
$\{\Vert \sigma_n\Vert_{BV}\}_{n=1}^\infty$ are uniformly bounded. 
Applying Lemma~\ref{lemma2} again, we derive that $\{ u_n\}_{n=1}^\infty$ is also uniformly bounded in $H^1(\Omega)$. 
Moreover, from the norm equivalence in Lemma~\ref{norm}, we deduce that
$\{\Vert u_n-U_{n,l}\Vert_{e_l}\}_{n=1}^\infty$ is bounded for $l=1,2,\ldots, L$. 
Thus, following the compact embedding of $BV(\Omega)$ in $L^1(\Omega)$ from Lemma~\ref{lsc_tv}, there exist $\sigma^\star\in \mathcal{A}$, $(u^\star, U^\star)\in \mathbb{H}$, and a subsequence, still denoted by $\{\sigma_n\}_{n=1}^\infty$, such that $\{\sigma_n\}_{n=1}^\infty$ converges to $\sigma^\star$ in $L^1$ norm, 
\beqnx\nb
&&u_n\to u^\star  \mbox{ weakly in } H^1(\Omega)\mbox{ and weakly in } L^2(\Omega),\\
&& \ u_n-U_{n,l}\to  u^\star-U^\star_l \mbox{ weakly in } L^2(e_l) \mbox{ and } U_{n,l}\to  U^\star_l \ \mbox{for}\ \mbox{all} \ l.
\eqnx
%
%
Next we prove that $(u^\star, U^\star)=\mathcal{F}(\sigma^\star) $.
By definition, $(u_n, U_n)= (u(\sigma_n),U(\sigma_n)) $ satisfies
\begin{equation}\label{seqeqn}
(\sigma_n\nabla u_n,\nabla v)+\sum_{l=1}^Lz_l^{-1}\langle u_n-U_{n, l}, v-V_l\rangle_{e_l}=\langle I, V\rangle\,   \quad \forall(v,V)\in \mathbb{H} .
\end{equation}
We will analyze the convergence of each term on the left-hand side of \eqref{seqeqn}. For the first term, we observe that 
\begin{eqnarray}
(\sigma_n\nabla u_n,\nabla v)-(\sigma^\star\nabla u^\star,\nabla v)&=&((\sigma_n-\sigma^\star)\nabla u_n,\nabla v)+(\sigma^\star\nabla(u_n-u^\star),\nabla v)\nonumber\\
&\leq&( \int_\Omega  |\sigma_n-\sigma^\star|\ |\nabla u_n|^2\, dx)^{1/2}( \int_\Omega  |\sigma_n-\sigma^\star|\ |\nabla v|^2\, dx)^{1/2} \nonumber\\
&&+(\sigma^\star\nabla(u_n-u^\star), \nabla v)\nonumber\\
&\leq&C( \int_\Omega  |\sigma_n-\sigma^\star|\ |\nabla v|^2\, dx)^{1/2}+(\sigma^\star\nabla(u_n-u^\star), \nabla v),\label{existineq}
\end{eqnarray}
where we have used $\sigma_n\in \mathcal{A}$ and that $\{ u_n\}_{n=1}^\infty$ is uniformly bounded in $H^1(\Omega)$. 
Since $\sigma_n\to \sigma^\star$ in $L^1(\Omega)$, there exists a subsequence of $\{\sigma_n\}_{n=1}^\infty$, still denoted by $\{\sigma_n\}_{n=1}^\infty$, that converges to $\sigma^\star$ pointwise almost everywhere.
By Lebesgue's dominated convergence theorem and the weak convergence of $\{u_n\}_{n=1}^\infty$,  we derive from \eqref{existineq} that
\beqn\nb
\lim_{n\to \infty}(\sigma_n\nabla u_n,\nabla v)=(\sigma^\star\nabla u^\star,\nabla v).
\eqn
Next, we consider the second term on the left-hand side of \eqref{seqeqn}. It follows from the weak convergence of $\{u_n-U_{n,l}\}_{n=1}^\infty$ that 
\beqn\nb
\lim_{n\to \infty}\sum_{l=1}^Lz_l^{-1}\la u_n-U_{n, l}, v-V_l\ra_{e_l}=\sum_{l=1}^Lz_l^{-1}\la u^\star-U_l^\star, v-V_l\ra_{e_l}.
\eqn
Upon taking into account these relations, one can deduce from \eqref{seqeqn} that
\begin{eqnarray*}
(\sigma^\star\nabla u^\star,\nabla v)+\sum_{l=1}^Lz_l^{-1}\la u^\star-U_l^\star, v-V_l\ra _{e_l}=\la I, V\ra,
\end{eqnarray*}
i.e.,  $(u^\star, U^\star)=F(\sigma^\star)$.
It remains to show that $\sigma^\star$ is indeed a minimizer of $J$.
Since $U_{n,l} \to U_l^\star$ in $L^2(e_l)$ for $l = 1,2,\ldots,L$, we have \begin{equation}\label{convU}
\lim_{n\to \infty}\frac{1}{2}\Vert U(\sigma_n)-U^\delta\Vert^2=\frac{1}{2}\Vert U(\sigma^\star)-U^\delta\Vert^2,
\end{equation}
which, together with Lemma~\ref{lsc_tv}, implies
\begin{eqnarray}
\nonumber J(\sigma^\star)&=&\frac{1}{2}\Vert U(\sigma^\star)-U^\delta\Vert^2+\alpha N(\sigma^\star)\\
\nonumber &=&\lim_{n\to \infty}\frac{1}{2}\Vert U(\sigma_n)-U^\delta\Vert^2+\alpha \int_\Omega|D\sigma^\star|\\
\nonumber &\leq&\lim_{n\to \infty}\frac{1}{2}\Vert U(\sigma_n)-U^\delta\Vert^2+\liminf\limits_{n\to\infty} \alpha \int_\Omega|D\sigma_n|\\
\nonumber &\leq&\liminf\limits_{n\to\infty} J(\sigma_n)\\
\nonumber &=&\inf_{\sigma \in \mathcal{A}}J(\sigma).
\end{eqnarray}
Hence $\sigma^\star$ is a minimizer of problem \eqref{mini}.
\end{proof}

The following theorem states that the solution to \eqref{mini} depends continuously on the data perturbation, that is,  the proposed regularized least-squares approach \eqref{mini} is well-posed.
\begin{thm} \label{THM3}
Let $\{U_n^\delta\}_{n=1}^\infty\subset \mathbb{R}^L_\diamond$ be a sequence of noisy data converging to $U^{\delta}$, and $\sigma_n$ be a minimizer to $J$ with  $U^\delta$ replaced by $U^\delta_n$ in \eqref{mini}. Then the sequence $\{\sigma_n\}_{n=1}^\infty$ has a subsequence converging  in $L^1(\Omega)$ to a minimizer of $J$.
\end{thm}
\begin{proof}
We denote by $J^n$ the functional with  $U^\delta$ replaced by $U_n^\delta$ in $J$, i.e.,
$$ J^n(\sigma) = \dfrac{1}{2}\Vert U(\sigma)-U^\delta_n\Vert^2+\alpha N(\sigma). $$
Since $\sigma_n$ minimizes the functional $J^n$ over $\mathcal{A}$, for all $n\geq 1$, we have 
$$J^n(\sigma_n) \leq J^n(1) \leq \| U(1) \|^2 + \sup_{m \geq 1} \| U^\delta_m \|^2,$$ 
where we have used the convergence and therefore the boundedness of $\{U_n^\delta\}_{n=1}^\infty$.  Consequently, we know that $\{N(\sigma_n)\}_{n=1}^\infty$ is bounded, and so is $\{\Vert \sigma_n\Vert_{BV(\Omega)}\}_{n=1}^\infty$. 
By the compact embedding of $BV(\Omega)$ in $L^1(\Omega)$ from Lemma~\ref{lsc_tv}, there exist $\sigma^\star \in \mathcal{A}$ and a subsequence of $\{\sigma_n\}_{n=1}^\infty$, still denoted by $\{\sigma_n\}_{n=1}^\infty$, such that $\sigma_n\to \sigma^\star$ in $L^1(\Omega)$ as $n\to\infty$. Now applying the  Cauchy-Schwarz inequality and the triangle inequality, we can deduce
\begin{align*}
|\Vert U(\sigma_n)-U^\delta_n\Vert^2-\Vert U(\sigma^\star)-U^\delta\Vert^2|
&=| \la U(\sigma_n)-U^\delta_n-U(\sigma^\star)+U^\delta, U(\sigma_n)-U^\delta_n+U(\sigma^\star)-U^\delta\ra |\\
&\leq ( \Vert U(\sigma_n)-U(\sigma^\star)\Vert+ \Vert U^\delta_n-U^\delta\Vert)\cdot \Vert U(\sigma_n)-U^\delta_n+U(\sigma^\star)-U^\delta\Vert\, .
\end{align*}
With Lemma \ref{staF}, one obtains that $\Vert U(\sigma_n)-U^\delta_n\Vert^2\to \Vert U(\sigma^\star)-U^\delta\Vert^2$  as $n\to\infty$. Then for any $\sigma \in \mathcal{A}$,  it follows Lemma~\ref{lsc_tv} that
\begin{align*}
J(\sigma^\star)
& = \frac{1}{2}\Vert U(\sigma^\star)-U^\delta\Vert^2 +\alpha N(\sigma^\star) \\
& \leq \lim_{n\to\infty}  \frac{1}{2}\Vert U(\sigma_n)-U^\delta_n\Vert^2 +\alpha\liminf\limits_{n\to\infty}  N(\sigma_n) \\
&\leq \liminf\limits_{n\to\infty}J^n(\sigma_n) \\
&\leq \liminf\limits_{n\to\infty}J^n(\sigma) \\
&=J(\sigma),
\end{align*}
which implies that $\sigma^\star$ is a minimizer of \eqref{mini}, completing the proof of this theorem.
\end{proof}

\section{Weak Galerkin method and error analysis}\label{sec4}

In this section, we propose the WG algorithm for discretizing the variational formulation \eqref{weakfor}, then develop the error analysis of this WG scheme  for the forward process.
 \subsection{WG formulation}
To discretize the variational formulation \eqref{weakfor}, we first triangulate the domain $\Omega$. Let $\mathcal{T}_h$ be the shape regular triangulation of the polyhedral domain $\overline{\Omega}$ consisting of closed simplicial elements, with a local mesh size $h_T:=|T|^{1/d}$ for each element $T\in \mathcal{T}_h$. We further assume that each element $T$ intersects at most one electrode surface $e_l$, and denote the mesh size of $\mathcal{T}_h$ by $h=\max\nolimits_{T\in\mathcal{T}_h} h_T$.  

We first recall an important concept of the WG method, weak function \cite{wang2013weak}. A weak function on a region $K$  refers to a function $v = \{v_0, v_b\}$, where $v_0 \in L^2(K)$ and $v_b\in L^2(\partial K)$. The first component
$v_0$ could be understood as the value of $v$ in $K$, and the second component $v_b$ represents $v$ on the
boundary of $K$, while $v_b$ may not necessarily be related to the trace of $v_0$ on $\partial K$ should a trace be well-defined. 
We define the space of weak functions on each $T\in \mathcal{T}_h$ by
\beqn\nonumber
S_w(T)=\{v=\{v_0,v_b\}: v_0\in L^2(T) ,\  v_b\in L^2(\partial T)\}\, ,
\eqn
and the finite element space on each $T$ by
\beqn\nb
S(k,T)=\{v_h = \{v_0, v_b\}: v_0\in P_k(T), v_b\in P_{k-1}(e) ,\ \mbox{edge (or face)} \ e\subset \partial T\}\, ,
\eqn
where the space $P_k(T)$ consists of all polynomials on the element $T$ with
degree not greater than $k$, and $P_{k-1}(e)$ consists of all polynomials on the edge (or face) $e\subset\partial T$ with
degree not greater than $k-1$. We take $k=1$ in this work, since the true solutions to \eqref{weakfor} have only limited regularity due to the complex physical setup of the CEM. Patching $S(1,T)$ together with a common value on $v_b$, we obtain the WG finite element space $S_h$ associated with $\mathcal{T}_h$ on the domain $\Omega$:
\beqn\nb
S_h=\{v_h =\{v_0,v_b\}: \{v_0,v_b\}|_T\in S(1,T),\ \forall T\in\mathcal{T}_h\}\, .
\eqn
 Then the test function space $\mathbb{H}$ for the variational equation \eqref{weakfor} can be approximated by $\mathbb{H}_h: = S_h\otimes \mathbb{R}_\diamond^L$.
Next, we present another key concept of WG method, the weak gradient operator $\nabla_w$.  For each $v_h\in S_h$, $\nabla_w v_h\in [P_{0}(T)]^d$ is defined as the unique polynomial satisfying the following equation:
\beqn\nb
(\nabla_w v_h, \mathbf{q})_T=\la v_b, \mathbf{q}\cdot \mathbf{n}\ra_{\partial T}\quad \forall \mathbf{q}\in  [P_{0}(T)]^d,
\eqn
where $(\cdot ,\cdot )_T$ denotes the  inner product on $[L^2(T)]^d$, and $\la\cdot ,\cdot \ra_{\partial T}$ denotes the inner product on $L^2(\partial T)$.  Using integration by parts, we have the following equivalent definition of the weak gradient operator $\nabla_w$: 
\begin{equation}\label{equivdef}
\begin{aligned}
(\nabla_w v_h, \mathbf{q})_T= (\nabla v_0, \mathbf{q})_T+\la v_b-v_0, \mathbf{q}\cdot \mathbf{n}\ra_{\partial T} \quad \forall \mathbf{q}\in  [P_{0}(T)]^d.
\end{aligned}
\end{equation}
To approximate the conductivity $\sigma$, we introduce the following standard piecewise constant finite element space over the triangulation $\mathcal{T}_h$:
\beqn\nb
W_h=\{\sigma_h: \sigma_h|_T\in P_0(T),\ \forall \, T\in\mathcal{T}_h\}\, ,
\eqn
and further define the discrete admissible set:
\beqn\nb
\mathcal{A}_{h}=\{\sigma_h\in W_h: \ \lambda\leq\sigma_h\leq \lambda^{-1} \text{ a.e. in } \Omega\},
\eqn
where $\lambda\in (0,1)$ represents the same constant in the definition \eqref{def:A} of the admissible set $\mathcal{A}$.
With these notations, we propose the WG finite element approximation of the variational equation \eqref{weakfor}.  For a given conductivity $\sigma_h \in \mathcal{A}_h$, we introduce the following two bilinear forms on $ \mathbb{H}_h\times \mathbb{H}_h$:
\beqn
\nonumber a_h(\sigma_h, (u_h, U_h), (v_h, V_h))&=&\sum_{T\in \mathcal{T}_h}(\sigma_h \nabla_w u_h, \nabla_w v_h)_T+\sum_{l=1}^Lz_l^{-1}\la u_b-U_{h,l}, v_b-V_{h,l}\ra_{e_l}\, ,\\
\nonumber s(\sigma_h, (u_h, U_h), (v_h, V_h))&=&\sum_{T\in\mathcal{T}_h}h_T^{-1}\la Q_b u_0-u_b, Q_b v_0-v_b\ra_{\partial T}\, ,
 \eqn
where $s$ is the stabilizer for the well-posedness of the WG algorithm, and $Q_b$ denotes the $L^2$ projection from $L^2(e)$ to $P_{0}(e)$ for the edge (or face) $e\subset\partial T$. Then the WG algorithm corresponding to the variational equation \eqref{weakfor} reads: 
\begin{algo}\label{algo1}
 Find $(u_h, U_h) \in \mathbb{H}_h$  satisfying
\beqn\label{WGfor}
a_s(\sigma_h, (u_h, U_h), (v_h, V_h))=\la I\, , V_h\ra\quad  \forall  (v_h, V_h)\in \mathbb{H}_h\, ,
\eqn
where  the operator $a_s$ is defined as 
\beqn\nonumber
a_s(\sigma_h, (u_h, U_h), (v_h, V_h))=a_h(\sigma_h, (u_h, U_h), (v_h, V_h))+s(\sigma_h, (u_h, U_h), (v_h, V_h)).
\eqn
\end{algo}
%
The following theorem states the unique existence of the solution to the WG formulation \eqref{WGfor}.
\begin{thm}\label{uniqsolver}
Given $\sigma_h\in\mathcal{A}_h$, the WG formulation \eqref{WGfor} has a unique solution in the finite element space $\mathbb{H}_h$.
\end{thm}
\begin{proof}
Since the equation system of \eqref{WGfor} has the same number of unknowns and equations, it suffices to prove the uniqueness of the solution to the system. We prove the uniqueness by contradiction. Assume that  $(u_h^1, U_h^1)$ and $(u_h^2, U_h^2)$ are two different solutions to \eqref{WGfor}. Denoting by $(\hat{u}_h, \hat{U}_h)$ the difference $(u_h^1, U_h^1)-(u_h^2, U_h^2)$,  we observe that  $(\hat{u}_h, \hat{U}_h)$ satisfies
\beqn\label{uniqsub}
\sum_{T\in\mathcal{T}}(\sigma_h \nabla_w \hat{u}_h, \nabla_w v_h)_T+\sum_{l=1}^Lz_l^{-1}\la\hat{u}_h-\hat{U}_l, v_h-V_l\ra_{e_l}+s(\sigma_h, (\hat{u}_h, \hat{U}_h), (v_h, V_h))=0\quad \forall (v_h, V_h)\in\mathbb{H}_h\, ,
\eqn
where $ \hat{U}_h=(\hat{U}_1,\ldots, \hat{U}_L)^t$. With the notation $\hat{u}_h = \{\hat{u}_0,\hat{u}_b\}$, taking $(v_h, V_h)=(\hat{u}_h, \hat{U}_h)$ in \eqref{uniqsub} yields \beqn\nb
  \nabla_w \hat{u}_h=0  \mbox{ on } T, \ \hat{u}_h=\hat{U}_h \mbox{ on } e_l,  \mbox{ and }   Q_b \hat{u}_0-\hat{u}_b =0 \mbox{ on } \partial T.
\eqn
Together with the definition \eqref{equivdef} of the weak gradient, we deduce that $\hat{u}_0 = \hat{u}_b$ on $ \partial T$ and $\hat{u}_0$ is a constant function on all elements $T$, hence all the entries of $\hat{U}_h\in \mathbb{R}_\diamond^L$ are equal to the constant. By the definition of $\mathbb{R}_\diamond^L$,  $\hat{U}_h=0$ and thus $(\hat{u}_h, \hat{U}_h)=0$, which contradicts the assumption that  $(u_h^1, U_h^1)$ and $(u_h^2, U_h^2)$ are two different solutions.  In this way we have proved the uniqueness of the solution.
\end{proof}
With Theorem \ref{uniqsolver} at hand, we can define the discrete forward operator $\mathcal{F}_h(\sigma_h) = (u_h(\sigma_h), U_h(\sigma_h))\in\mathbb{H}_h$ that maps $\sigma_h\in\mathcal{A}_h$ to the unique solution $(u_h,U_h)$ to \eqref{WGfor} as a discrete analogue of the forward operator $\mathcal{F}$.

\subsection{$L^2$ projections and approximation properties}
We define the local projection operators on the WG finite element spaces and derive some approximation properties which are useful in the convergence analysis.
For each element $T\in \mathcal{T}_h$, we denote the $L^2$ projection from $L^2(T)$ to $P_1(T)$ by $Q_0$. Recall that the $L^2$ projection from $L^2(e)$ to $P_{0}(e)$ is denoted by $Q_b$  for edge (or face) $e$. At the same time, the $L^2$ projection from $[L^2(T)]^d$ to the local discrete gradient space $[P_{0}(T)]^d$ will be denoted by $\mathbb{Q}_h$. 
We further introduce a projection operator $Q_h: H^1(\Omega)\to S_h$ so that    on each element $T\in \mathcal{T}_h$, 
\beqn\nb
Q_h v=\{Q_0 v_0, Q_b v_b \}\, ,\ \{v_0, v_b\}=i_w(v)\in S_w(T)\, ,
\eqn
 where the inclusion map $i_w: H^1(T)\to S_w(T)$ is defined as
\beqn\nb
i_w(\phi)=\{\phi|_T, \phi|_{\partial T}\}\, ,\ \phi\in H^1(T)\, .
\eqn
The following lemma in \cite[Lemma 5.1]{mu2015weak} states that the projection operators commute with the differential operators. 
\begin{lem} \label{conm}
On each element $T \in \mathcal{T}_h$ we have the following commutative property
\begin{equation}\label{comm}
\nabla_w(Q_h \phi)=\mathbb{Q}_h(\nabla \phi) \quad \forall \phi\in H^1(T)\, .
\end{equation}
\end{lem}
Moreover, we have the following approximation properties for $Q_0$ and $\mathbb{Q}_h$ \cite[Lemma 5.2]{mu2015weak}.
\begin{lem}
\label{Q0} 
Assume that  $\phi \in H^1(\Omega)$ and $\phi|_T\in H^2(T)$  for all $T\in \mathcal{T}_h$. 
Then for $0 \leq s \leq 1$, we have 
\begin{align}
\sum_{T\in \mathcal{T}_h}\Vert \phi-Q_0\phi\Vert_{L^2(T)}^2+\sum_{T\in \mathcal{T}_h} h_T^2\Vert \nabla(\phi-Q_0\phi)\Vert_{[L^2(T)]^d}^2
&\leq Ch^{2s+2}\sum_{T\in\mathcal{T}_h}\Vert \phi\Vert_{H^{s+1}(T)}^2, \label{1lemma4.1} \\
\sum_{T\in \mathcal{T}_h} \Vert \nabla \phi-\mathbb{Q}_h(\nabla \phi)\Vert_{[L^2(T)]^d}^2
&\leq Ch^{2s}\sum_{T\in\mathcal{T}_h}\Vert \phi\Vert_{H^{s+1}(T)}^2. \label{2lemma4.1}
\end{align}
\end{lem}
To study the approximation estimates of these projection operators on the edge (or face) of elements, we need the trace inequality in \cite{wang2014weak}:
let $T$ be an element with $e\subset\partial T$ an edge (or face), then for any function $\psi\in H^1(T)$, 
\begin{eqnarray}
\Vert \psi\Vert_{L^2(e)}^2\leq C(h_T^{-1}\Vert \psi\Vert_{L^2(T)}^2+h_T\Vert \nabla \psi\Vert_{[L^2(T)]^d}^2)\, .\label{trace}
\end{eqnarray}
With the trace inequality \eqref{trace} and the estimates in Lemma~\ref{Q0}, it is straightforward to 
derive the following estimates of the projection operators $Q_0$ and $\mathbb{Q}_h$ on $\partial T$.
\begin{lem}
\label{Q1}
Under the assumptions in Lemma \ref{Q0},
for $0 \leq s \leq 1$, we have 
\begin{align}
\sum_{T\in \mathcal{T}_h} \Vert \phi-Q_0\phi\Vert_{L^2(\partial T)}^2
&\leq Ch^{2s+1}\sum_{T\in\mathcal{T}_h}\Vert \phi\Vert_{H^{s+1}(T)}^2,\label{3lemma4.1} \\
\sum_{T\in \mathcal{T}_h} \Vert (\nabla \phi-\mathbb{Q}_h(\nabla \phi)) \cdot \mathbf{n} \Vert_{L^2(\partial T)}^2
&\leq Ch\sum_{T\in\mathcal{T}_h}\Vert \phi\Vert_{H^{2}(T)}^2. \label{4lemma4.1}
\end{align}
\end{lem}
We also present the following useful estimate for the projection operator $Q_b$.
\begin{lem} \label{Qb}
Under the same assumptions as in Lemma \ref{Q0}, we have
\begin{equation}
\label{lemma4.4}
(\sum_{T\in\mathcal{T}_h} \Vert \phi-Q_b \phi\Vert^2_{L^2(\partial T)})^{1/2}\leq Ch^{1/2}\Vert \phi\Vert_{H^1(\Omega)}\, .
\end{equation}
\end{lem}
\begin{proof}
By definition, $Q_b$ is the $L^2$ projection from $L^2(e)$  to $P_0(e)$. Then we define $Q_{0, 0}$ as the $ L^2$ projection from $L^2(T)$ to $P_0(T)$, which leads to 
\beqn\nonumber
\Vert \phi-Q_b \phi\Vert_{L^2(\partial T)}\leq \Vert \phi-Q_{0,0} \phi\Vert_{L^2(\partial T)}\, .
\eqn
From the trace inequality \eqref{trace} and the approximation property of $Q_{0,0}$, we also have
\beqn\nb
 (\sum_{T\in\mathcal{T}_h}\Vert \phi-Q_{0, 0}\phi\Vert_{L^2(\partial T)}^2)^{1/2}\leq Ch^{1/2}\Vert \phi\Vert_{H^1(\Omega)}\, ,
 \eqn
and the desired inequality \eqref{lemma4.4} follows.
\end{proof}

\subsection{Error estimates for the WG algorithm}
Next, we analyze the error between 
the finite element solution $(u_h, U_h)$ to the WG formulation \eqref{WGfor} and 
the exact solution $(u, U)$ to the system of equations \eqref{contieqn}. 
To simplify the notations, we define the error quantities 
 $$(e_h, E_h): =(u_h-Q_h u, U_h-U)\in\mathbb{H}_h$$
 and denote $e_h = \{e_0, e_b\}$, $E_h = (E_1,\ldots, E_L)^t$.
 We first derive the corresponding error equation for the WG algorithm \ref{algo1}. 
\begin{lem}\label{erroreqlemma}
Assume that $u\in H^1(\Omega)$,  $u|_T \in H^{2}(T)$, and $\sigma\nabla u|_T\in [H^1(T)]^d$ for all $T \in \mathcal{T}_h$ when mesh size $h\leq h_0$. 
Let $(e_h, E_h) =(u_h-Q_h u, U_h-U)$. Then for any $ (v_h, V_h)\in \mathbb{H}_h$, 
\begin{equation}
\begin{aligned}
 a_s(\sigma_h, (e_h, E_h),(v_h, V_h))=&\sum_{T\in\mathcal{T}_h} \la (\sigma _h\mathbb{Q}_h (\nabla u)-\sigma\nabla u)\cdot\mathbf{n}, v_0-v_b\ra_{\partial T}+\sum_{T\in\mathcal{T}_h} ((\sigma-\sigma_h) \nabla  u, \nabla v_0)_T\\
&-s(\sigma_h, (Q_h u, U),( v_h, V_h)).\label{erroreqn}
\end{aligned}
\end{equation}
\end{lem}
\begin{proof}
By Lemma \ref{conm} and the definition of the weak gradient operator,  for any $\phi\in H^1(\Omega)$ and $v_h\in S_h $,
\begin{equation}
\begin{split}
(\sigma_h\nabla_w Q_h\phi, \nabla_w v_h)_T&=(\sigma_h \mathbb{Q}_h \nabla\phi, \nabla_w v_h)_T\\
&=((\sigma_h-\sigma)\nabla\phi, \nabla v_0)_T+(\sigma\nabla\phi, \nabla v_0)_T-\la \sigma_h \mathbb{Q}_h(\nabla\phi)\cdot\mathbf{n}, v_0-v_b\ra_{\partial T}\, .
\end{split}
\label{error1}
\end{equation}
By testing the first equation in \eqref{contieqn} with $v_0$ and summing up over $T\in\mathcal{T}_h$, we can deduce from the integration by parts  and boundary conditions in \eqref{contieqn} that
\begin{equation}
\begin{split}
0&=\sum_{T\in\mathcal{T}_h}  \int_T (-\nabla\cdot (\sigma\nabla u(x)) v_0(x) dx\\
&=-\sum_{T\in\mathcal{T}_h}\la\sigma\nabla u\cdot \mathbf{n},  v_0\ra_{\partial T}+\sum_{T\in\mathcal{T}_h} (\sigma \nabla u, \nabla v_0)_T\\
&=-\sum_{T\in\mathcal{T}_h}\la  \sigma\nabla u\cdot \mathbf{n}, v_0-v_b\ra_{\partial T}+\sum_{T\in\mathcal{T}_h} (\sigma \nabla u, \nabla v_0)_T
+\sum_{l=1}^L z_l^{-1}\la u-U_l,v_b\ra_{e_l}\, .
\end{split}
\label{error2}
\end{equation}
Taking $\phi = u$ in \eqref{error1},  one can derive from \eqref{error1} and \eqref{error2} that
\begin{eqnarray}
0&=&\sum_{T\in\mathcal{T}_h} (\sigma_h \nabla_w Q_h u, \nabla_w v_h)_T+\sum_{T\in\mathcal{T}_h} \la \sigma_h \mathbb{Q}_h (\nabla u)\cdot\mathbf{n}, v_0-v_b\ra_{\partial T}\nonumber\\
&&+\sum_{T\in\mathcal{T}_h}((\sigma-\sigma_h)\nabla u,\nabla v_0)_T -\sum_{T\in\mathcal{T}_h}\la \sigma\nabla u\cdot \mathbf{n}, v_0-v_b \ra_{\partial T}+\sum_{l=1}^L z_l^{-1}\la  u-U_l,v_b\ra_{e_l}\, ,\label{zero2}
\end{eqnarray}
which, together with the boundary conditions in \eqref{contieqn}, leads to \begin{equation}
\label{discrete_weakfor}
\begin{split}
\la I, V_h\ra&=\sum_{T\in\mathcal{T}_h} (\sigma_h \nabla_w Q_h u, \nabla_w v_h)_T+\sum_{T\in\mathcal{T}_h} \la \sigma_h \mathbb{Q}_h (\nabla u)\cdot\mathbf{n}, v_0-v_b\ra_{\partial T}\\
&\quad+\sum_{T\in\mathcal{T}_h}((\sigma-\sigma_h)\nabla u,\nabla v_0)_T-\sum_{T\in\mathcal{T}_h}\la \sigma\nabla u\cdot \mathbf{n}, v_0-v_b \ra_{\partial T} \\
&\quad +\sum_{l=1}^L z_l^{-1}\la  Q_b u-U_l, v_b-V_l\ra_{e_l}\\
&=a_h(\sigma_h, (Q_h u, U), (v_h, V_h))+\sum_{T\in\mathcal{T}_h} \la (\sigma_h\mathbb{Q}_h (\nabla u)-\sigma\nabla u)\cdot\mathbf{n},  v_0-v_b\ra_{\partial T}\\
&\quad+\sum_{T\in\mathcal{T}_h}((\sigma-\sigma_h)\nabla u,\nabla v_0)_T\,  .
\end{split}
\end{equation}
Combining \eqref{WGfor} and \eqref{discrete_weakfor}, we arrive at the following error equation of $(e_h, E_h)$,
\begin{equation*}
\begin{aligned}
 a_s(\sigma_h, (e_h, E_h),(v_h, V_h))=&\sum_{T\in\mathcal{T}_h} \la (\sigma _h\mathbb{Q}_h (\nabla u)-\sigma\nabla u)\cdot\mathbf{n}, v_0-v_b\ra_{\partial T}+\sum_{T\in\mathcal{T}_h} ((\sigma-\sigma_h) \nabla  u, \nabla v_0)_T\\
&-s(\sigma_h, (Q_h u, U),( v_h, V_h))\, 
\end{aligned}
\end{equation*}
for all $  (v_h, V_h)\in \mathbb{H}_h$. This completes the derivation of \eqref{erroreqn}.
\end{proof}
%

Next, to prepare for the error estimate, we further introduce the following norms on the space $\mathbb{H}_h$ in analogue to $\Vert \cdot\Vert_h$ on the space $\mathbb{H}$.  For $(v_h, V_h)\in \mathbb{H}_h$, we define 
\beqn \nb
\Vert (v_h, V_h)\Vert^2_{1,h}&=&\sum_{T\in\mathcal{T}_h}(\Vert \nabla v_0\Vert_{[L^2(T)]^d}^2+ h_T^{-1}\Vert Q_b v_0-v_b\Vert^2_{L^2(\partial T)})+\sum_{l=1}^L \Vert v_b-V_{h,l}\Vert_{L^2(e_l)}^2\, ,\\
||| (v_h, V_h)|||_h^2&=&a_s(1, (v_h, V_h), (v_h, V_h))\, . \nonumber
\eqn
The equivalence of these two norms is given below.
\begin{lem} \label{lemma6}
There exist positive constants $C_1, C_2$ such that for any $(v_h, V_h)\in \mathbb{H}_h$, 
\beqn\nb
C_1\Vert (v_h, V_h)\Vert^2_{1,h}\leq |||( v_h, V_h)|||_h^2\leq C_2\Vert (v_h, V_h)\Vert^2_{1,h}\, .
\eqn
\end{lem}
\begin{proof}
For any $v_h=\{v_0, v_b\}\in S_h$, by the definition \eqref{equivdef} of the weak gradient operator, we have
\begin{equation}\label{lemma_6}
(\nabla_w v_h, \mathbf{q})_T=(\nabla v_0, \mathbf{q})_T+\la v_b-Q_b v_0, \mathbf{q}\cdot \mathbf{n}\ra_{\partial T}\quad\forall  \mathbf{q}\in [P_{0}(T)]^d\,.
\end{equation}
Taking $\mathbf{q}=\nabla_w v_h$ in \eqref{lemma_6} yields
\beqn\nb
(\nabla_w v_h, \nabla_w v_h)_T=(\nabla v_0, \nabla_w v_h)_T +\la v_b-Q_b v_0, \nabla_w v_h\cdot \mathbf{n}\ra_{\partial T}\, .
\eqn
This together with the  Cauchy-Schwarz inequality and trace inequality \eqref{trace} gives 
\begin{eqnarray}
\nonumber (\nabla_w v_h, \nabla_w v_h)_T& \leq & \Vert \nabla v_0\Vert_{[L^2(T)]^d}\Vert  \nabla_w v_h\Vert_{[L^2(T)]^d}+\Vert Q_b v_0-v_b\Vert_{L^2(\partial T)}\Vert \nabla_w v_h\cdot \mathbf{n}\Vert_{L^2(\partial T)}\\ 
\nonumber &\leq &\Vert \nabla v_0\Vert_{[L^2(T)]^d} \Vert \nabla_w v_h\Vert_{[L^2(T)]^d}+Ch_T^{-1/2}\Vert Q_b v_0-v_b\Vert_{L^2(\partial T)}\Vert \nabla_w v_h\Vert_{[L^2(T)]^d}\, .
\end{eqnarray}
Then we deduce that
\beqn\nb
\Vert \nabla_w v_h\Vert_{[L^2(T)]^d} \leq C(\Vert \nabla v_0\Vert_{[L^2(T)]^d}^2+h_T^{-1}\Vert Q_b v_0-v_b\Vert_{L^2(\partial T)}^2)^{1/2}\, ,
\eqn
which implies
\beqn\label{upperequiv}
 |||(v_h, V_h)|||_h^2\leq C_2 \Vert(v_h, V_h)\Vert_{1,h}^2\, .
 \eqn
Next, taking $q=\nabla v_0$ in \eqref{lemma_6}, we obtain
\beqn\nb
(\nabla _w v_h, \nabla v_0)_T=(\nabla v_0, \nabla v_0)_T+\la v_b-Q_b v_0, \nabla v_0\cdot\mathbf{n}\ra_{\partial T}\, .
\eqn
Further by the  Cauchy-Schwarz inequality and trace inequality \eqref{trace} again, we arrive at 
\beqn \nb
\Vert \nabla v_0\Vert_{[L^2(T)]^d}\leq C(\Vert \nabla_w v_h\Vert_{[L^2(T)]^d}^2+Ch_T^{-1}\Vert Q_b v_0-v_b\Vert_{L^2(\partial T)}^2)^{1/2}\, ,
\eqn
which leads to 
\beqn\nb
C_1\Vert (v_h, V_h)\Vert^2_{1,h}\leq |||( v_h, V_h)|||_h^2
\eqn
and completes the proof together with \eqref{upperequiv}.
\end{proof}

Now we present the following error estimate of the WG algorithm \ref{algo1} in the induced norm $|||\cdot |||_h$.
\begin{thm} \label{tribnorm}
Under the same regularity assumptions as in Lemma~\ref{erroreqlemma}, there holds
\begin{eqnarray*}
&&||| ( e_h, E_h) |||_h \\
&&\leq C  \Bigg(h\left(\sum_{T\in\mathcal{T}_h}\Vert u\Vert_{H^2(T)}^2\right)^{1/2}+h \left(\sum_{T\in\mathcal{T}_h}\Vert \sigma\nabla u\Vert_{[H^{1}(T)]^d}^2\right)^{1/2}+\left(\sum_{T \in \mathcal{T}_h}\Vert (\sigma-\sigma_h)\nabla u\Vert_{[L^2(T)]^d}^2\right)^{1/2}\Bigg).
\end{eqnarray*}

\end{thm}
\begin{proof}
Taking $(v_h, V_h)=(e_h, E_h)$ in the error equation \eqref{erroreqn} yields 
\begin{equation}
\begin{split}
a_s(\sigma_h, (e_h, E_h),(e_h, E_h))=&\sum_{T\in\mathcal{T}_h} \la (\sigma_h \mathbb{Q}_h(\nabla u)-\sigma \nabla u)\cdot\mathbf{n}, e_0-e_b\ra_{\partial T}\\ 
& \quad +\sum_{T\in\mathcal{T}_h} ((\sigma-\sigma_h) \nabla u, \nabla e_0)_T
 -s(\sigma_h, (Q_h u, U),( e_h, E_h)).
\end{split}
\label{thm5.1}
\end{equation}
We shall estimate each of these terms on the right-hand side of \eqref{thm5.1}. 
For the first term, it follows from the Cauchy–Schwarz inequality, the approximation property \eqref{4lemma4.1}, and the trace inequality \eqref{trace} that
\begin{align*}
&\left|\sum_{T\in\mathcal{T}_h} \la (\sigma_h\mathbb{Q}_h (\nabla u)-\sigma\nabla u)\cdot\mathbf{n}, e_0-e_b\ra_{\partial T}\right|\\
\leq \; &\sum_{T\in\mathcal{T}_h} \left|\la \sigma_h (\mathbb{Q}_h (\nabla u)-\nabla u)\cdot\mathbf{n}, e_0-e_b\ra_{\partial T}
+\la (-\sigma+\sigma_h) \nabla u\cdot\mathbf{n}, e_0-e_b\ra_{\partial T}\right|\\
\leq \; & \sum_{T\in\mathcal{T}_h} (\Vert \sigma_h (\mathbb{Q}_h (\nabla u)-\nabla u)\cdot\mathbf{n}\Vert _{L^2(\partial T)} +\Vert (\sigma-\sigma_h) \nabla u\cdot \mathbf{n}\Vert_{L^2(\partial T)})\Vert e_0-e_b\Vert_{L^2(\partial T)}\\
\leq \; &C \left(\left(\sum_{T\in\mathcal{T}_h} h_T\Vert \sigma_h (\mathbb{Q}_h (\nabla u)-\nabla u)\cdot\mathbf{n}\Vert ^2_{L^2(\partial T)}\right)^{1/2}
+\left(\sum_{T\in\mathcal{T}_h} h_T\Vert (\sigma-\sigma_h)\nabla u\cdot \mathbf{n}\Vert ^2_{L^2(\partial T)}\right)^{1/2}\right)
\\
&\cdot\left(\sum_{T\in\mathcal{T}_h} h_T^{-1}\Vert e_0-e_b\Vert^2_{L^2(\partial T)}\right)^{1/2}\\
\leq \; & C\left( h \left(\sum_{T\in\mathcal{T}_h}\Vert u\Vert_{H^{2}(T)}^2\right)^{1/2} +h \left(\sum_{T\in\mathcal{T}_h}\Vert \sigma\nabla u\Vert_{[H^1(T)]^d}^2\right)^{1/2}
+  \left(\sum_{T\in\mathcal{T}_h}\Vert  (\sigma-\sigma_h)\nabla u\Vert_{[L^2(T)]^d}^2\right)^{1/2}\right)\\
&\cdot
\left(\sum_{T\in\mathcal{T}_h} h_T^{-1}\Vert e_0-e_b\Vert^2_{L^2(\partial T)}\right)^{1/2}\, ,
\end{align*}
where we have used the fact that $\sigma_h$ is piecewise constant and the regularity assumptions on $u$ and $\sigma\nabla u$. %
 Following the approximation property of the $L^2$ projection operator $Q_b$ and the trace inequality \eqref{trace}, we also have
\begin{eqnarray*}
\Vert e_0-e_b\Vert_{L^2(\partial T)}&\leq \Vert e_0-Q_b e_0\Vert_{L^2(\partial T)}+\Vert Q_b e_0-e_b\Vert_{L^2(\partial T)}\\
& \leq  Ch_T^{1/2}\Vert \nabla e_0\Vert_{[L^2(T)]^d}+\Vert Q_b e_0-e_b\Vert_{L^2(\partial T)}\, .
\end{eqnarray*}
Together with Lemma \ref{lemma6} this leads to
\begin{equation*}
\sum_{T\in\mathcal{T}_h} h_T^{-1}\Vert e_0-e_b\Vert_{L^2(\partial T)}^2\leq C \left\| (e_h, E_h) \right\|_{1,h}^2 \leq C |||(e_h, E_h)|||^2_h\, .
\end{equation*}
Thus we obtain 
\begin{equation}
\begin{split}
&\left|\sum_{T\in\mathcal{T}_h} \la (\sigma_h\mathbb{Q}_h (\nabla u)-\sigma\nabla u)\cdot\mathbf{n}, e_0-e_b\ra_{\partial T}\right|\\
\leq \; & C\left( h \left(\sum_{T\in\mathcal{T}_h}\Vert u\Vert_{H^{2}(T)}^2\right)^{1/2} +h \left(\sum_{T\in\mathcal{T}_h}\Vert \sigma\nabla u\Vert_{[H^1(T)]^d}^2\right)^{1/2}
+\left(\sum_{T\in\mathcal{T}_h}\Vert  (\sigma-\sigma_h)\nabla u\Vert_{[L^2(T)]^d}^2\right)^{1/2}\right)\\
&\cdot|||(e_h, E_h)|||_h.
 \end{split}
 \label{thm5.1.1}
 \end{equation}
By Cauchy–Schwarz inequality and the norm equivalence in Lemma~\ref{lemma6}, the second term on the right-hand side of \eqref{thm5.1} satisfies
\begin{equation}
\begin{split}
\sum_{T\in\mathcal{T}_h} ((\sigma-\sigma_h) \nabla  u, \nabla e_0)_T&\leq
 \left(\sum_{T\in\mathcal{T}_h}\Vert (\sigma-\sigma_h) \nabla u\Vert^2_{[L^2(T)]^d}\right)^{1/2}\left(\sum_{T\in\mathcal{T}_h}  \Vert \nabla e_0\Vert^2_{[L^2(T)]^d} \right)^{1/2} \\
&\leq \left(\sum_{T\in\mathcal{T}_h}\Vert (\sigma-\sigma_h) \nabla u\Vert^2_{[L^2(T)]^d}\right)^{1/2}|||(e_h, E_h)|||_h\, .
 \end{split}
 \label{thm5.1.2}
 \end{equation}
 Finally, for the third term on the right-hand side of \eqref{thm5.1}, it follows from the approximation property \eqref{3lemma4.1} and Lemma~\ref{lemma6} that
\begin{equation}
\begin{split}
|s(\sigma_h, (Q_h u, U),( e_h, E_h))|
&=\left|\sum_{T\in\mathcal{T}_h} h_T^{-1}\la  Q_b(Q_0 u)-Q_b u, Q_b e_0-e_b\ra_{\partial T}\right|\\
&\leq \sum_{T\in\mathcal{T}_h} h_T^{-1} \Vert Q_b( Q_0 u-u) \Vert_{L^2(\partial T)} \Vert Q_b e_0-e_b\Vert_{L^2(\partial T)}\\
&\leq\left(\sum_{T\in\mathcal{T}_h} h_T^{-1} \Vert Q_0 u-u \Vert^2_{L^2(\partial T)}\right)^{1/2}
\left(\sum_{T\in\mathcal{T}_h} h_T^{-1} \Vert Q_b e_0-e_b\Vert^2_{L^2(\partial T)} \right)^{1/2}\\
&\leq C h \left(\sum_{T\in\mathcal{T}_h}\Vert u\Vert_{H^{2}(T)}^2\right)^{1/2} |||(e_h, E_h)|||_h.
 \end{split}
 \label{thm5.1.3}
 \end{equation}
 Substituting \eqref{thm5.1.1}--\eqref{thm5.1.3} in \eqref{thm5.1}, we obtain the desired result.
\end{proof}

While Theorem \ref{tribnorm} above provides an error estimate for the WG algorithm \ref{algo1} in the induced norm $|||\cdot|||_h$, it is crucial to analyze the convergence of the finite element approximation of the electrode voltage $U$ in the Euclidean norm, which is the measurement discrepancy term in the objective functional of the minimization problem \eqref{mini}.
To this end, we derive the error estimate for the electrode voltage $U$ in the Euclidean norm using a duality argument. 
Consider a dual problem that seeks $\psi\in H^1(\Omega)$ and $\Psi\in \mathbb{R}^L_\diamond$  satisfying
\begin{equation}\label{dual1}
\left\{
\begin{array}{rlll}
-\nabla\cdot(\sigma\nabla \psi)&=&0& \mbox{ in } \Omega\, ,\\
\psi+z_l\sigma\frac{\partial \psi}{\partial n}&=&\Psi_l& \mbox{ on }e_l  \mbox{ for } l=1,2,\ldots, L\, ,\\
\int_{e_l}\sigma\frac{\partial \psi}{\partial n}\, ds&=&E_{h,l}& 
\mbox{ for } l=1,2,\ldots, L\, ,\\
\sigma\frac{\partial \psi}{\partial n}&=&0& \mbox{ on } \Gamma-\cup_{l=1}^L e_l\, .
\end{array}
\right.
\end{equation}
From Lemma \ref{lemma2} and Lemma \ref{staF}, the dual problem has the $H^{1}$-regularity and $\Vert \psi\Vert_{H^1(\Omega)}\leq C\Vert E_h\Vert$. We will further assume that $\psi|_T \in H^{2}(T)$ and $\sigma\nabla \psi|_T\in [H^1(T)]^d$ for all $T \in \mathcal{T}_h$ when mesh size $h\leq h_0$, and 
  \begin{equation}\label{assumpdual}
  \sum_{T\in\mathcal{T}_h}\Vert \psi\Vert_{H^2(T)}+\sum_{T\in\mathcal{T}_h}\Vert \sigma\nabla\psi\Vert_{[H^1(T)]^d}\leq C\Vert E_h\Vert.
  \end{equation}

The next theorem provides the error estimate for electrode voltage $U$.
\begin{thm} \label{errorU}
Let the regularity assumptions in Lemma~\ref{erroreqlemma}  and  \eqref{assumpdual} hold. Then we have 
\begin{eqnarray*}
\| E_h \| \leq C\Bigg(h \left(\sum_{T\in\mathcal{T}_h}\Vert u\Vert^2_{H^{2}(T)}\right)^{1/2}+h \left(\sum_{T\in\mathcal{T}_h}\Vert \sigma\nabla u\Vert_{[H^{1}(T)]^d}^2\right)^{1/2}+\left(\sum_{T\in\mathcal{T}_h} \Vert(\sigma-\sigma_h) \nabla u\Vert^2_{[L^2(T)]^d}\right)^{1/2}\Bigg).
\end{eqnarray*}
\end{thm}
\begin{proof}
Testing the first equation of the dual problem \eqref{dual1} with $e_0$ on each element and applying integration by parts, we obtain \begin{equation}
\begin{split}
0&=\sum_{T\in\mathcal{T}_h} \int_T (-\nabla\cdot(\sigma\nabla \psi(x))\cdot e_0(x)dx\\
 &=-\sum_{T\in\mathcal{T}_h} \la \sigma\nabla \psi\cdot \mathbf{n}, e_0 \ra_{\partial T}+\sum_{T\in\mathcal{T}_h}(\sigma\nabla \psi, \nabla e_0)_T\\
&=-\sum_{T\in\mathcal{T}_h} \la \sigma\nabla \psi\cdot \mathbf{n}, e_0-e_b \ra_{\partial T}+\sum_{T\in\mathcal{T}_h}(\sigma\nabla \psi, \nabla e_0)_T-\sum_{l=1}^L \la \sigma\nabla \psi\cdot \mathbf{n}, e_b \ra_{e_l}\, .\label{l2error1}
\end{split}
\end{equation}
Taking $\phi=\psi$, $v_h=e_h$ in the equality \eqref{error1} yields \begin{equation}
(\sigma_h\nabla_w Q_h\psi, \nabla_w e_h)_T+\la \sigma_h \mathbb{Q}_h(\nabla\psi)\cdot\mathbf{n}, e_0-e_b\ra_{\partial T}+((\sigma-\sigma_h)\nabla\psi,\nabla e_0)_T
=(\sigma\nabla\psi, \nabla e_0)_T\, .\label{l2error2}
\end{equation}
With \eqref{l2error2} and the boundary conditions in \eqref{dual1}, we can further deduce from \eqref{l2error1} that 
\begin{equation}
\label{l2error3}
\begin{split}
  \sum_{l=1}^L \la E_{h,l}, E_{h,l}\ra_{e_l}=&\sum_{T\in\mathcal{T}_h} \la (\sigma_h \mathbb{Q}_h(\nabla\psi)-\sigma\nabla \psi)\cdot \mathbf{n}, e_0-e_b \ra_{\partial T}+\sum_{T\in\mathcal{T}_h} (\sigma_h\nabla_w Q_h\psi, \nabla_w e_h)_T\\
  & +\sum_{l=1}^Lz_l^{-1}\la \psi-\Psi, e_b-E_h \ra_{e_l}
 +\sum_{T\in\mathcal{T}_h}((\sigma-\sigma_h)\nabla \psi,\nabla e_0)_T .
 \end{split}
\end{equation}
Recall that $(e_h, E_h)$ satisfies the error equation \eqref{erroreqn}. Taking $(v_h, V_h) = (Q_h\psi, \Psi) \in \mathbb{H}_h$ in  \eqref{erroreqn} yields
\beqn
\begin{aligned}
\nonumber(\sigma_h\nabla_w Q_h\psi, \nabla_w e_h)=&\sum_{T\in\mathcal{T}_h} \la (\sigma_h\mathbb{Q}_h (\nabla u)-\sigma\nabla u)\cdot\mathbf{n}, Q_0 \psi-Q_b\psi\ra_{\partial T}+\sum_{T\in\mathcal{T}_h} ((\sigma-\sigma_h) \nabla  u, \nabla Q_0 \psi)_T\\
\nonumber &-s(\sigma_h, ( u_h, U_h),( Q_h\psi, \Psi))-\sum_{l=1}^L z_l^{-1}(Q_b\psi-\Psi, e_b-E_h)_{e_l}\, .\label{l2error4}
\end{aligned}
\eqn
Thus we can rewrite \eqref{l2error3} as
\begin{align}
 \nonumber \sum_{l=1}^L \la E_{h,l}, E_{h,l}\ra_{e_l}=&\sum_{T\in\mathcal{T}_h} \la \sigma_h(\mathbb{Q}_h(\nabla\psi)-\nabla \psi)\cdot \mathbf{n}, e_0-e_b \ra_{\partial T}+\sum_{T\in\mathcal{T}_h} \la \sigma_h(\mathbb{Q}_h (\nabla u)-\nabla u)\cdot\mathbf{n}, Q_0 \psi-Q_b\psi\ra_{\partial T}\\
 \nonumber&+\sum_{T\in\mathcal{T}_h} \la (\sigma_h-\sigma)\nabla \psi\cdot \mathbf{n}, e_0-e_b \ra_{\partial T}+\sum_{T\in\mathcal{T}_h} \la (\sigma_h-\sigma)\nabla u\cdot \mathbf{n}, Q_0\psi-Q_b\psi \ra_{\partial T} \\
 &-s(\sigma_h, ( u_h, U_h),( Q_h\psi, \Psi))+\sum_{T\in\mathcal{T}_h} ((\sigma-\sigma_h) \nabla  u, \nabla  Q_0 \psi)_T+\sum_{T\in\mathcal{T}_h}((\sigma-\sigma_h)\nabla\psi,\nabla e_0)_T\, .\label{error}
 \end{align}
Next, we estimate each of these terms on the right-hand side of \eqref{error}. 
For the first term, it follows from \eqref{4lemma4.1} and Lemma~\ref{lemma6} that 
\begin{equation}
\begin{split}
& \sum_{T\in\mathcal{T}_h} \la \sigma_h( \mathbb{Q}_h(\nabla\psi)-\nabla \psi)\cdot \mathbf{n}, e_0-e_b \ra_{\partial T} \\
\leq \; & \sum_{T\in\mathcal{T}_h} \Vert \sigma_h (\mathbb{Q}_h (\nabla \psi)-\nabla \psi) \cdot \mathbf{n}\Vert_{L^2(\partial T)} \Vert e_0-e_b\Vert_{L^2(\partial T)}\\
\leq \; & C \left(\sum_{T\in\mathcal{T}_h} h_T\Vert (\mathbb{Q}_h (\nabla \psi)-\nabla \psi) \cdot \mathbf{n} \Vert ^2_{L^2(\partial T)}\right)^{1/2} \left(\sum_{T\in\mathcal{T}_h} h_T^{-1}\Vert e_0-e_b\Vert^2_{L^2(\partial T)}\right)^{1/2}\\
\leq \; &  C h \left( \sum_{T\in\mathcal{T}_h} \Vert\psi\Vert^2_{H^{2}(T)} \right)^{1/2}|||(e_h, E_h)|||_h\, . 
\end{split}
\label{error4.1}
\end{equation}
Similarly, for the second term, 
we can deduce from the approximation properties \eqref{3lemma4.1}, \eqref{4lemma4.1}, and \eqref{lemma4.4} that 
\begin{equation}
\begin{split}
& \sum_{T\in\mathcal{T}_h} \la \sigma_h (\mathbb{Q}_h (\nabla u)-\nabla u)\cdot\mathbf{n}, Q_0 \psi-Q_b\psi\ra_{\partial T} \\
\leq \; & \sum_{T\in\mathcal{T}_h} \Vert \sigma_h (\mathbb{Q}_h (\nabla u)-\nabla u) \cdot \mathbf{n}\Vert _{L^2(\partial T)} \Vert Q_0 \psi-Q_b\psi\Vert_{L^2(\partial T)}\\
\leq \; & C \left(\sum_{T\in\mathcal{T}_h} h_T\Vert (\mathbb{Q}_h (\nabla u)-\nabla u) \cdot \mathbf{n} \Vert ^2_{L^2(\partial T)}\right)^{1/2} \left(\sum_{T\in\mathcal{T}_h} h_T^{-1} \Vert Q_0 \psi - Q_b \psi \Vert^2_{L^2(\partial T)}\right)^{1/2}\\
\leq \; &  C h \left( \sum_{T\in\mathcal{T}_h} \Vert u\Vert^2_{H^{2}(T)} \right)^{1/2} \| \psi \|_{H^1(\Omega)}. 
\end{split}
\label{error4.2}
\end{equation}
 By the trace inequality \eqref{trace} and Lemma~\ref{lemma6},
 the third and forth terms on the right-hand side of \eqref{error} satisfy
 \begin{equation}\label{error4.3}
\begin{split}
&\sum_{T\in\mathcal{T}_h} \la (\sigma_h-\sigma)\nabla \psi\cdot \mathbf{n}, e_0-e_b \ra_{\partial T} \\
&\leq C\left(\left(\sum_{T\in\mathcal{T}_h}\Vert (\sigma_h-\sigma)\nabla \psi\cdot \mathbf{n}\Vert^2_{L^2(T)}\right)^{1/2}+  h \left(\sum_{T\in\mathcal{T}_h}\Vert (\sigma_h-\sigma)\nabla \psi\Vert_{[H^{1}(T)]^d}^2\right)^{1/2}\right)\cdot|||(e_h, E_h)|||_h\, ,\\
&\sum_{T\in\mathcal{T}_h} \la (\sigma_h-\sigma)\nabla u\cdot \mathbf{n}, Q_0\psi-Q_b\psi \ra_{\partial T}\\
&\leq C \left(\left(\sum_{T\in\mathcal{T}_h}\Vert(\sigma_h-\sigma)\nabla u\cdot \mathbf{n}\Vert^2_{L^2(T)}\right)^{1/2}  +h \left(\sum_{T\in\mathcal{T}_h}\Vert (\sigma_h-\sigma)\nabla u\Vert_{[H^{1}(T)]^d}^2\right)^{1/2}\right)\cdot \Vert\psi\Vert_{H^1(\Omega)}\, ,
\end{split}
\end{equation}
where we have used the regularity assumptions for the dual problem \eqref{dual1}.
By the triangle inequality, we observe that the fifth term on the right-hand side of \eqref{error} satisfies
\begin{equation*}
|s(\sigma_h, ( u_h, U_h),( Q_h\psi, \Psi))|\leq |s(\sigma_h, ( e_h, E_h),( Q_h\psi, \Psi))|+|s(\sigma_h, ( Q_h u, U),( Q_h\psi, \Psi))|\, ,
\end{equation*}
and it further admits upper bound using the approximation properties \eqref{3lemma4.1}, \eqref{4lemma4.1}  and Lemma~\ref{lemma6}:
\begin{equation}\label{error4.4}
|s(\sigma_h, ( e_h, E_h),( Q_h\psi, \Psi))|\leq C h \left(\sum_{T\in\mathcal{T}}\Vert \psi\Vert_{H^2(T)}^2\right)^{1/2} |||(e_h, E_h)|||_h\,,
\end{equation}
\begin{equation}
\begin{split}
|s(\sigma_h, ( Q_h u, U),( Q_h\psi, \Psi))|
&\leq\sum_{T\in\mathcal{T}_h} h_T^{-1}|\la Q_b(Q_0 u)-Q_b u, Q_b(Q_0\psi)-Q_b \psi\ra_{\partial T}|\\
&\leq\sum_{T\in\mathcal{T}_h} h_T^{-1}\Vert Q_b(Q_0 u - u)\Vert_{L^2(\partial T)}\Vert Q_b(Q_0\psi - \psi)\Vert_{L^2(\partial T)}\\
&\leq C\sum_{T\in\mathcal{T}_h} h_T^{-1}\Vert Q_0 u- u\Vert_{L^2(\partial T)}\Vert Q_0\psi-\psi\Vert_{L^2(\partial T)}\\
&\leq C\left(\sum_{T\in\mathcal{T}_h} h_T^{-1}\Vert Q_0 u- u\Vert^2_{L^2(\partial T)}\right)^{1/2}\left(\sum_{T\in\mathcal{T}_h} h_T^{-1}\Vert Q_0\psi-\psi\Vert^2_{L^2(\partial T)}\right)^{1/2}\\
&\leq Ch^2 \left(\sum_{T\in\mathcal{T}_h} \Vert u\Vert^2_{H^{2}(T)}\right)^{1/2}  \left(\sum_{T\in\mathcal{T}_h}\Vert \psi\Vert^2_{H^{2}(T)}\right)^{1/2}\, .
\end{split}
\label{error4.5}
\end{equation}
Finally, for the last two terms on the right-hand side of \eqref{error}, we employ  \eqref{1lemma4.1} and the norm equivalence in Lemma~\ref{lemma6} to conclude
\begin{equation}
\begin{split}
\sum_{T\in\mathcal{T}_h} ((\sigma-\sigma_h) \nabla u, \nabla Q_0 \psi)_T
&\leq C\left(\sum_{T\in\mathcal{T}_h} \Vert(\sigma-\sigma_h) \nabla u\Vert_{[L^2(T)]^d}^2\right)^{1/2} \Vert \psi\Vert_{H^1(\Omega)},\\
\sum_{T\in\mathcal{T}_h}((\sigma-\sigma_h)\nabla\psi,\nabla e_0)_T&\leq C |||(e_h, E_h)|||_h\Vert\psi\Vert_{H^1(\Omega)}.
\end{split}
\label{error4.6}
\end{equation}
Substituting \eqref{error4.1}--\eqref{error4.6} into \eqref{error} and applying Theorem~\ref{tribnorm} yield
\begin{eqnarray*}
 \Vert E_h\Vert^2&\leq & C\left(h \left(\sum_{T\in\mathcal{T}_h}\Vert u\Vert^2_{H^{2}(T)}\right)^{1/2}+h \left(\sum_{T\in\mathcal{T}_h}\Vert \sigma\nabla u\Vert_{[H^{1}(T)]^d}^2\right)^{1/2}+ \left(\sum_{T\in\mathcal{T}_h}\Vert  (\sigma-\sigma_h)\nabla u\Vert_{[L^2(T)]^d}^2\right)^{1/2}\right)\\
&&\cdot\left(\Vert \psi\Vert_{H^1(\Omega)}+h \left(\sum_{T\in\mathcal{T}_h}\Vert \psi\Vert^2_{H^{2}(T)}\right)^{1/2}+  h \left(\sum_{T\in\mathcal{T}_h}\Vert \sigma\nabla \psi\Vert_{[H^{1}(T)]^d}^2\right)^{1/2}\right).
\end{eqnarray*}
Together with the regularity of the dual problem \eqref{dual1} and the assumption \eqref{assumpdual}, this leads to the desired result
\begin{eqnarray*}
 \Vert E_h\Vert&\leq& C\left(h \left(\sum_{T\in\mathcal{T}_h}\Vert u\Vert^2_{H^{2}(T)}\right)^{1/2}+h \left(\sum_{T\in\mathcal{T}_h}\Vert \sigma\nabla u\Vert_{[H^{1}(T)]^d}^2\right)^{1/2}+ \left(\sum_{T\in\mathcal{T}_h}\Vert  (\sigma-\sigma_h)\nabla u\Vert_{[L^2(T)]^d}^2\right)^{1/2}\right).
\end{eqnarray*}

\end{proof}

\section{{Convergence analysis}}\label{sec5}

In this section, we discretize the regularized optimality system \eqref{mini} and establish the convergence of the discrete minimizers utilizing the error estimates derived in Section \ref{sec4}.

To prepare for the discretization, we first present the following result \cite[Theorem 3.1]{asas1999regularization} concerning the space of piecewise constant functions, $W_h$. It states that $W_h$ is a subspace of $BV(\Omega)$ and 
provides an explicit formula for the total variation of the piecewise constant functions.
\begin{lem}\label{disBV}
For any $\sigma_h \in W_h$, we have $\sigma_h \in BV(\Omega)$ and 
\begin{equation}
\int_\Omega \vert D \sigma_h \vert = \dfrac{1}{2} \sum_{T_1, T_2\in\mathcal{T}_h} |\sigma_{h,1}-\sigma_{h,2}||\partial T_1\cap \partial T_2|,
\end{equation}
where $\sigma_{h,i}$ is the value of $\sigma_h$ on $T_i$.
\end{lem}
 We also state the following approximation property \cite[Theorem 3.4]{asas1999regularization}.
\begin{lem}\label{constBV}
For every $\sigma\in \mathcal{A}$ there exists a sequence of $\{\sigma_h \}$ with  $\sigma_h\in \mathcal{A}_h$  such that
\beqn\nb
\lim_{h\to 0}\int_\Omega |\sigma-\sigma_h|\, dx=0\ \mbox{ and }\ \lim_{h\to 0}\int_{\Omega} \vert D \sigma_h \vert =\int_\Omega |D\sigma|\, .
\eqn
\end{lem}
Thus the discrete analogue of the minimization problem \eqref{mini} reads: 
\begin{equation}\label{DM}
\min_{\sigma_h \in \mathcal{A}_h} 
\left\{J_h(\sigma_h)=\frac{1}{2}\Vert U_h(\sigma_h)-U^{\delta}\Vert^2+\alpha N_h(\sigma_h)\right\},
\end{equation}
where
\beqn \nb
N_h(\sigma_h)=\dfrac{1}{2} \sum_{T_1, T_2\in\mathcal{T}_h} |\sigma_{h,1}-\sigma_{h,2}||\partial T_1\cap \partial T_2|\, .
\eqn
Next we present the existence and stability of the solution to  \eqref{DM} with respect to the measurement data.
The proofs are identical to those presented for Theorem~\ref{exist} and Theorem~\ref{THM3}, and thus omitted for clarity.
\begin{thm}
There exists at least one solution to the discrete minimization problem \eqref{DM}.
\end{thm}
\begin{thm}
 Let $\{U_n^\delta\}_{n=1}^\infty\subset \mathbb{R}^L_\diamond$ be a sequence of noisy data converging to $U^{\delta}$, and $\sigma_h^n$ be a minimizer to $J_h$ with $U^\delta_n$ in place of $U^\delta$. Then the sequence $\{\sigma^n_h\}_{n=1}^\infty$ has a subsequence converging to a minimizer of $J_h$.
\end{thm}

The remaining of this section is devoted to establishing the convergence of the solution to \eqref{DM}.  More precisely, we will prove that a sequence of minimizers to the discrete minimization problems \eqref{DM} will converge subsequentially to a minimizer to the continuous minimization problem \eqref{mini} as the mesh size turns to zero.

\begin{thm}
Let $\{\sigma_{h_k}^\star\}_{k=1}^\infty$ be a sequence of minimizers of the discrete minimization problems \eqref{DM} with mesh size $h_k$, where $h_k\to 0$ as $k\to\infty$. Under the assumptions of Theorem~\ref{errorU}, there exists a subsequence of $\{\sigma_{h_k}^\star\}_{k=1}^\infty$ converging in $L^1(\Omega)$ to a minimizer of the continuous problem \eqref{mini}.
\end{thm}
\begin{proof}
By the minimizing property of $\sigma_{h_k}^\star$ to the functional $J_{h_k}$ over $\mathcal{A}_{h_k}$,  we have 
$$J_h(\sigma_{h_k}^\star) \leq J_{h_k}(1) = \dfrac{1}{2} \| U_{h_k}(1) - U^\delta \|^2.$$ 
Therefore $\{N_{h_k}(\sigma_{h_k})\}_{k=1}^\infty$ is bounded and hence the sequence $\{\sigma_{h_k}^\star\}_{k=1}^\infty$ is bounded in BV norm following Lemma~\ref{disBV}. 
By Lemma~\ref{lsc_tv}, there exist $\sigma^\star \in \mathcal{A}$ and a subsequence, still denoted by $\{\sigma_{h_k}^\star\}_{k=1}^\infty$, such that $\sigma_{h_k}^\star \to \sigma^\star$ in $L^1(\Omega)$ as $k\to \infty$. 
We claim that  $\sigma^\star$ is a minimizer of \eqref{mini}.
For any $\sigma\in \mathcal{A}$, Lemma~\ref{constBV} implies that there exists a sequence $\{\sigma_{h_k}\}_{k=1}^\infty$ with  $\sigma_{h_k}\in \mathcal{A}_{h_k}$ such  that
\begin{equation}\label{seq}
\lim_{k\to \infty}\int_\Omega |\sigma-\sigma_{h_k}|\, dx=0\ \mbox{and}\ \lim_{k\to \infty}N_{h_k}(\sigma_{h_k})=N(\sigma).
\end{equation}
Again, by the minimizing property  of $\{\sigma_h^\star\}$, there holds
\begin{equation}\label{mi}
J_h(\sigma_h^\star)\leq  J_h(\sigma_h)\, .
\end{equation}
Using Theorem~\ref{errorU} and Lemma~\ref{lsc_tv}, together with \eqref{seq} and \eqref{mi}, we deduce that
\begin{align*}
J(\sigma^\star) &=\frac{1}{2}\Vert U(\sigma^\star)-U^\delta\Vert^2+\alpha N(\sigma^\star)\\
&\leq \frac{1}{2}\lim_{h\to 0} \Vert U_h(\sigma^\star_h)-U^\delta\Vert^2+\alpha\liminf\limits_{h\to 0}   N_h(\sigma_h^\star)\\
&\leq \liminf\limits_{h\to 0} J_h(\sigma_h^\star)\\
&\leq \liminf\limits_{h\to 0} J_h(\sigma_h)\\
&= \lim_{h\to 0} \left(\frac{1}{2}\Vert U_h(\sigma_h)-U^\delta\Vert^2+\alpha N_h(\sigma_h)\right)\\
&=\frac{1}{2}\Vert U(\sigma)-U^\delta\Vert^2+\alpha N(\sigma)\\
&=J(\sigma)\, .
\end{align*}
Hence, we conclude that  $\sigma^\star$ is a minimizer of problem \eqref{mini}.
\end{proof}

\begin{section}{Numerical experiments}\label{sec6}

In this section, we  present some numerical examples to showcase the WG method for the forward process and the proposed BV-based regularization approach for the inverse process in EIT. All the computations were carried out using MATLAB 2018b on a personal laptop with 8.00 GB RAM and 2.7 GHz CPU. The setup of these numerical experiments is as follows. We take the domain $\Omega$  as a square $(0,1)^2$. 
There are $16$ electrodes $\{e_l\}_{l=1}^L$ ($L=16$) evenly distributed along the boundary $\Gamma$, each of length $1/8$. 
We set all the contact impedances $\{z_l\}_{l=1}^L$ to unit and the background conductivity $\sigma_0\equiv 1$. 

\subsection{Experiment 1: Convergence rate of WG method}

In this experiment, we examine the convergence of the WG method for the process. We set up a model problem \eqref{contieqn} with given physical data, i.e. conductivity field $\sigma^\dag$, unit contact impedance $\{z_l\}_{l=1}^L$ and input current $I$, and solve the problem numerically by the WG algorithm \ref{algo1} at difference mesh size $h$.
\begin{exam}\label{ex:1}
The exact conductivity is given by $\sigma^\dag\equiv 1$ and the input current $I$ is sinusoidal, i.e., $I_i =\sin(i\pi/4)$, $i=1,2,...,16.$
  \end{exam}

Figure~\ref{forward} depicts the numerical solutions $u_h$ to the WG algorithm \ref{algo1} with the triangulations of mesh size $h = 1/16$ and $h = 1/64$. 
Since there is no close formula for the analytical solution to this model problem \eqref{contieqn}, we take the numerical solution with mesh size $h=1/128$ as the reference solution. 
Table~\ref{conv} records the error at different mesh size $h$ and the convergence rate. 
It is noted that the convergence rate of the WG algorithm is at least $O(h)$ in error of both potential $u$ and electrode voltage $U$, which verifies our theoretical results.

\begin{figure}[ht!]
\centering
\subfigure[mesh size $h=1/16$]{
\label{ c}
\includegraphics[width=0.4\textwidth]{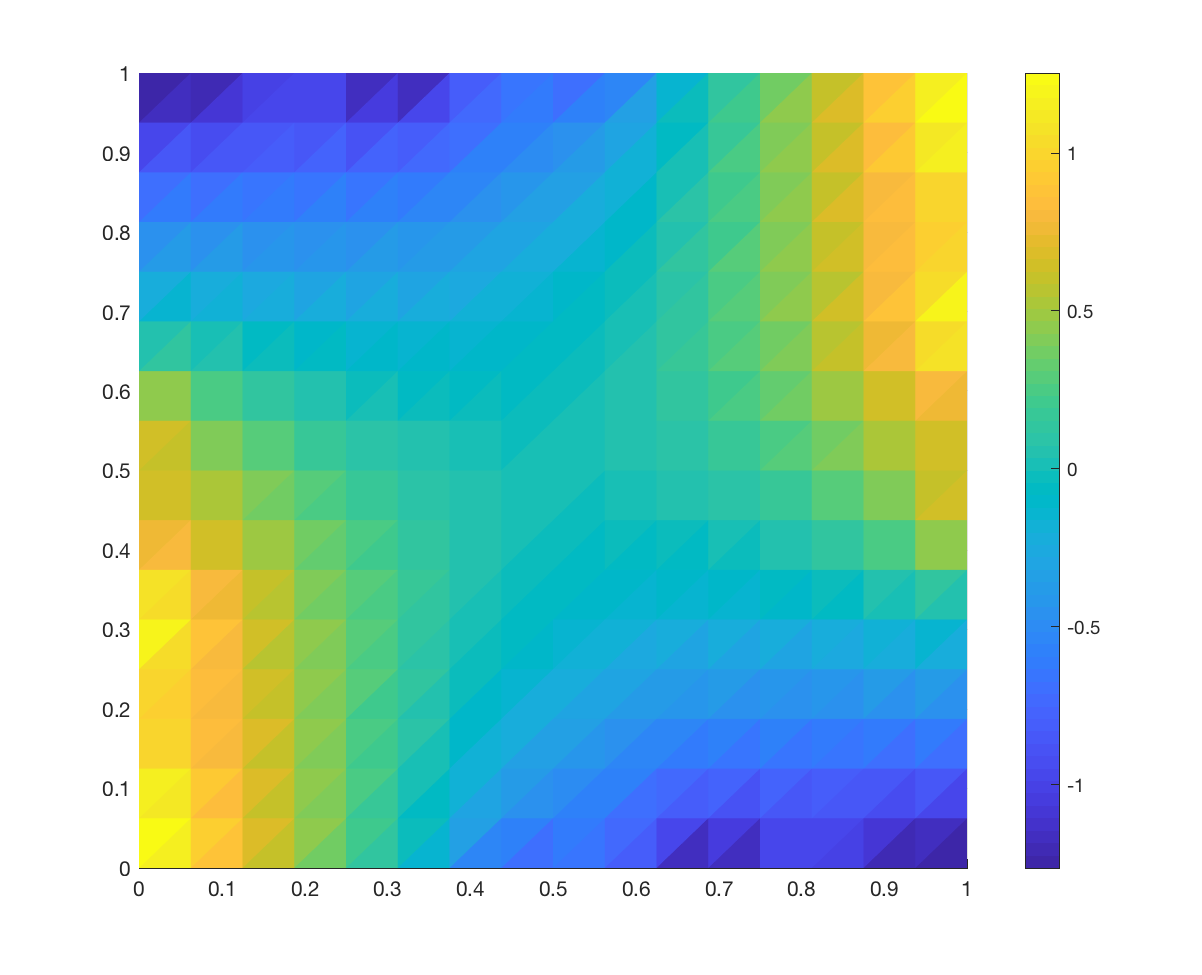}}
\subfigure[mesh size $h=1/64$]{
\label{Fig.sub.2}
\includegraphics[width=0.4\textwidth]{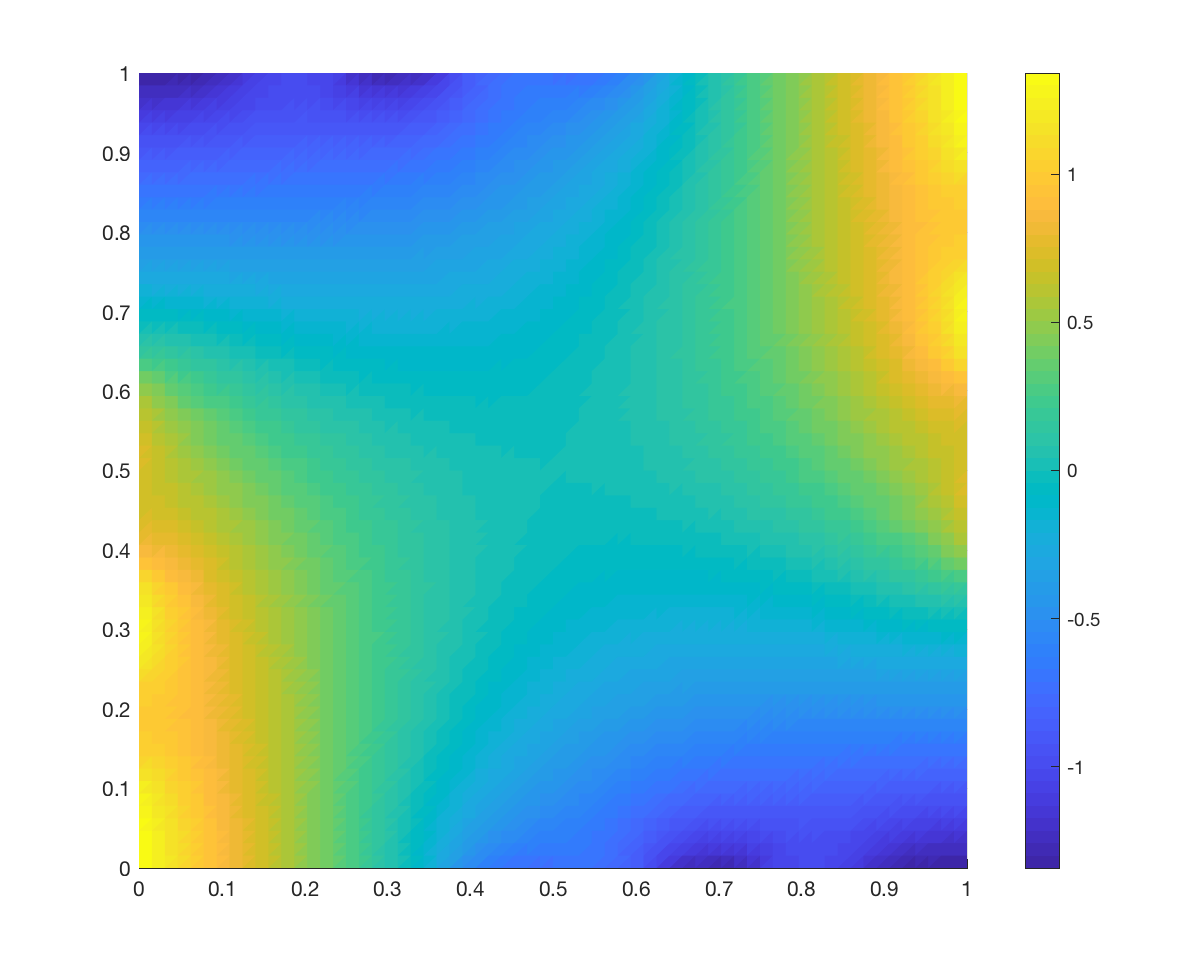}}
\caption{Approximated interior voltage $u_h$ using WG method}
\label{forward}
\end{figure}


\begin{table}[!htbp]
\centering
\begin{tabular}{c||c|c||c|c}
  $h$&$\Vert e_h\Vert_{L^2(\Omega)}$& order &$\Vert E_h \Vert $ & order \\
\hline
\hline
1/8  & $1.39  \times 10^{-1}  $& -- &$6.45\times 10^{-1}$ & -- \\
\hline
1/16 & $ 6.75\times 10^{-2}$& 1.0376 & $2.25 \times 10^{-2}$ &  1.5225 \\
\hline
1/32 & $3.28\times 10^{-2}$ &  1.0438 &$6.88\times 10^{-2}$ &  1.7075 \\
\hline
1/64 & $1.46 \times 10^{-2}$&  1.1637 &$ 1.65\times 10^{-2}$ &  2.0629 \\
\end{tabular}
\caption{History of convergence of WG methods}\label{conv}
\end{table}

\subsection{Experiment 2: Reconstruction of conductivity}

In this experiment, we present several numerical
examples to illustrate the effectiveness of the proposed BV-based least-squares approach  \eqref{DM} for EIT.  
The electrode voltages $U$ are generated and measured for $10$ times corresponding to ten sinusoidal input currents to gain enough information for the sought-for conductivity $\sigma^\dag$. 
In each example, we will generate the exact data $U(\sigma^\dag)$ on a mesh that is much finer than the mesh used for reconstruction to avoid the ``inverse crime''. 
To generate the noisy data $U^\delta$, we add component wise Gaussian noise to the exact data $U(\sigma^\dag)$ as follows:
\beqn\nb
U_l^\delta=U_l(\sigma^\dag)+\epsilon\max_j |U_j(\sigma^\dag)|\xi_l,\ l=1,2,...,L \, ,
\eqn
where $\{\xi_l\}$ is taken following the standard normal distribution and $\epsilon$ is the (relative) noise level. 
In these examples, the regularization parameter $\alpha$ is taken in a trial-and-error manner, which suffices
our goal of illustrating the significant potentials of the proposed approach \eqref{DM} for EIT.

Here we briefly describe the numerical algorithm for the discrete minimization problem \eqref{DM}. 
It is noted that the objective functional $J_h$ in \eqref{DM} is non-differentiable. 
To this end, we introduce the Fast Iterative Shrinkage-Thresholding Algorithm (FISTA) \cite{beck2009fast}  to minimize the objective functional $J_h$. The basic idea of FISTA is to perform a gradient update on the differentiable part (the measurement discrepancy), take the image under a proximal map for the non-differentiable part (the total variation), and update the solution using the previous two iterations.  To calculate the G\^{a}teaux derivative in different directions for the measurement discrepancy term in $J_h$, we will use the WG method to solve an auxiliary dual problem derived from \eqref{WGfor}.
The detailed procedure is explained in Appendix~\ref{app1}. 
\begin{exam}\label{ex:2}
The exact conductivity is a linear function given by $\sigma^\dag(x,y) = \sigma_0+\frac{2}{3} x$ in $\Omega$ with the background conductivity $\sigma_0\equiv 0.5$. 
\end{exam}
In this example, we reconstruct the conductivity using exact measurement and noisy measurement with noise level $\epsilon = 0.1\%$ of the electrode voltage respectively, and compare the profiles in these two scenarios. The initial guess for the iterations is $\sigma_0\equiv 0.5$ in $\Omega$. The true conductivity $\sigma^\dag$ and the reconstruction $\sigma_h^\star$ with   mesh size $h = 1/64$ are depicted in Figure~\ref{example1}. 
It is observed that the reconstruction captures both the magnitude and shape of the exact conductivity very accurately, and for such smooth conductivity, this approach provides almost identical profiles, as shown in Figure \ref{li.sub.2} and \ref{li.sub.3},  for either the exact or the noisy data, with only slight difference near the boundary,  which justifies the robustness of the proposed approach thanks to the BV regularization.

In Figure~\ref{example1conv}, we plot the $L^2(\Omega)$ error of the reconstruction $\sigma_h^\star$ versus the mesh size $h$. The $L^2(\Omega)$ error for the reconstruction with noiseless and noisy measurements are plotted in blue and red respectively. 
It is observed  that the error $\|\sigma_h^\star - \sigma^\dag\|_{L^2(\Omega)}$ has a linear convergence in both cases, which justifies the effectiveness and the robustness of our approach with respect to measurements. 

\begin{figure}[ht!]
\centering
\subfigure[True conductivity $\sigma^\dag$]{
\label{li.sub.1}
\includegraphics[width=0.32\textwidth]{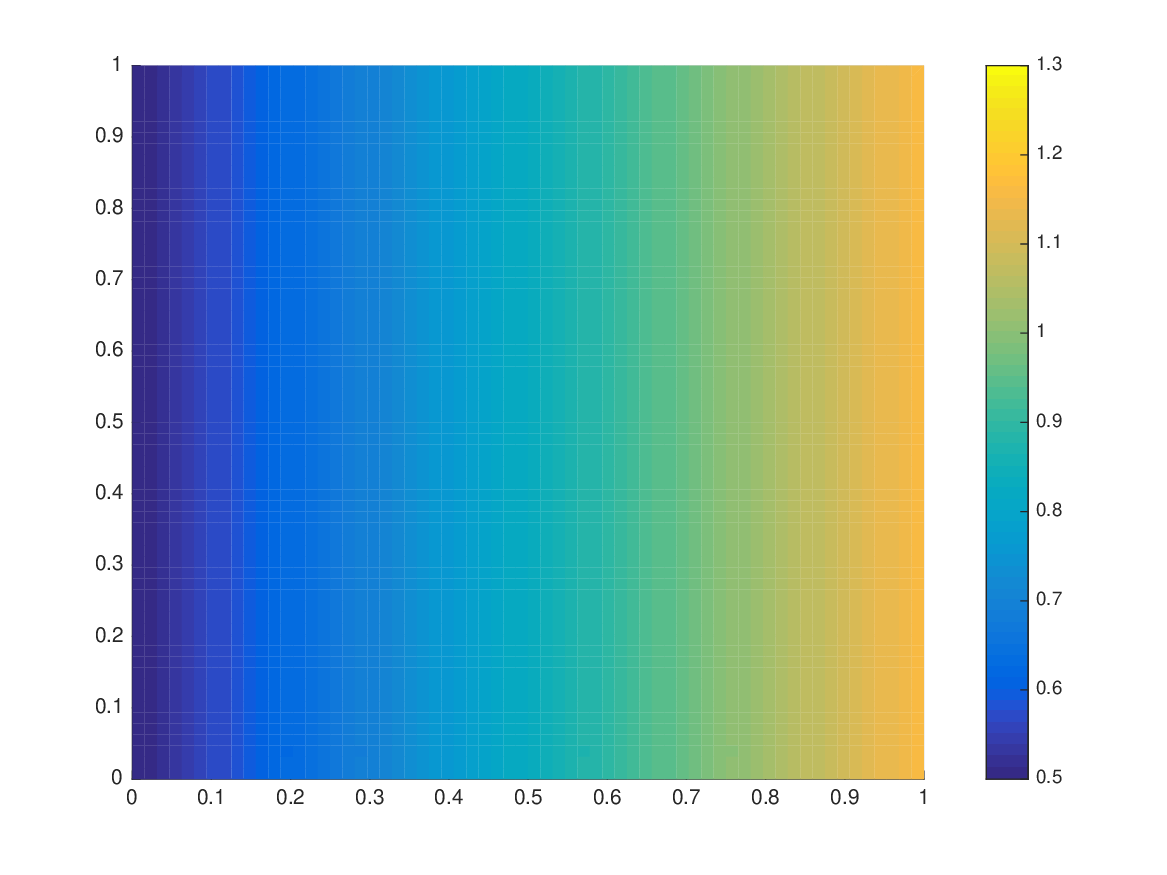}}
\subfigure[Reconstruction $\sigma_h^\star$ with noiseless measurements]{
\label{li.sub.2}
\includegraphics[width=0.32\textwidth]{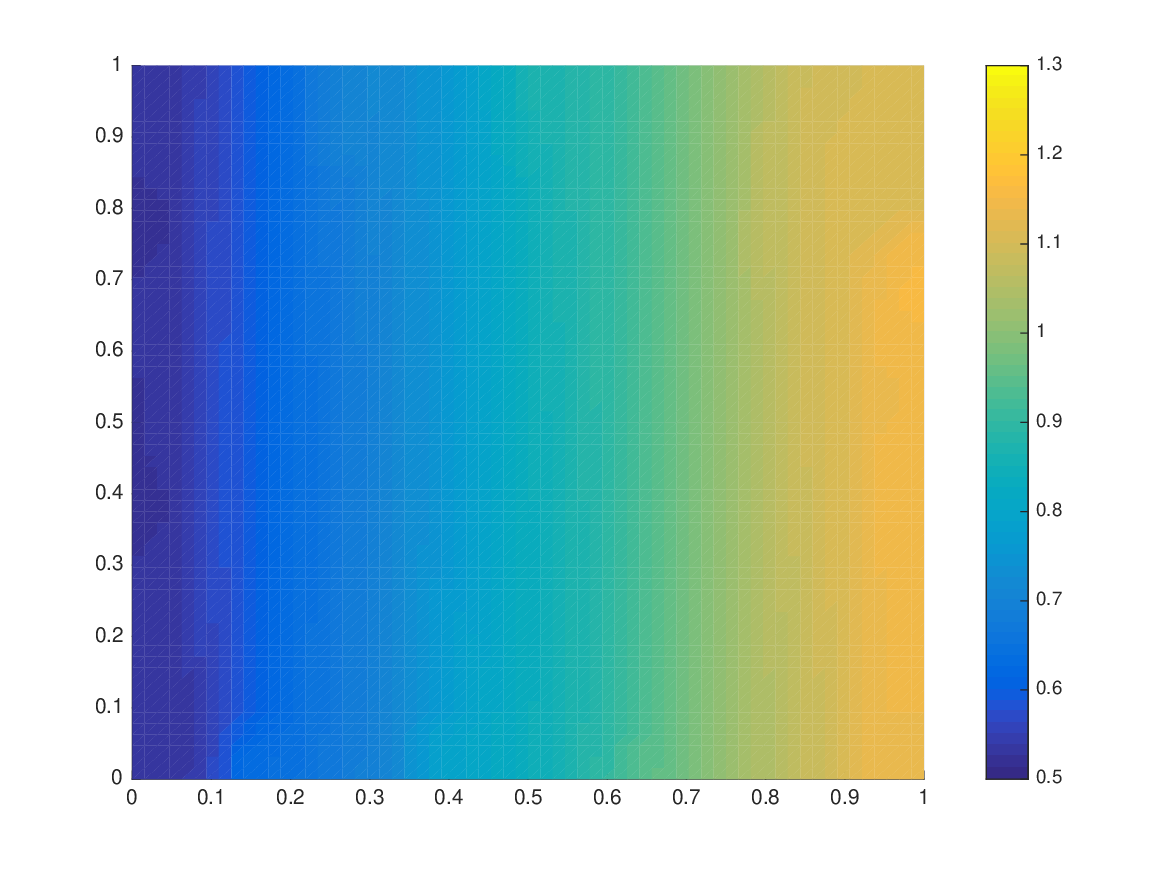}}
\subfigure[Reconstruction $\sigma_h^\star$ with noisy measurements]{
\label{li.sub.3}
\includegraphics[width=0.32\textwidth]{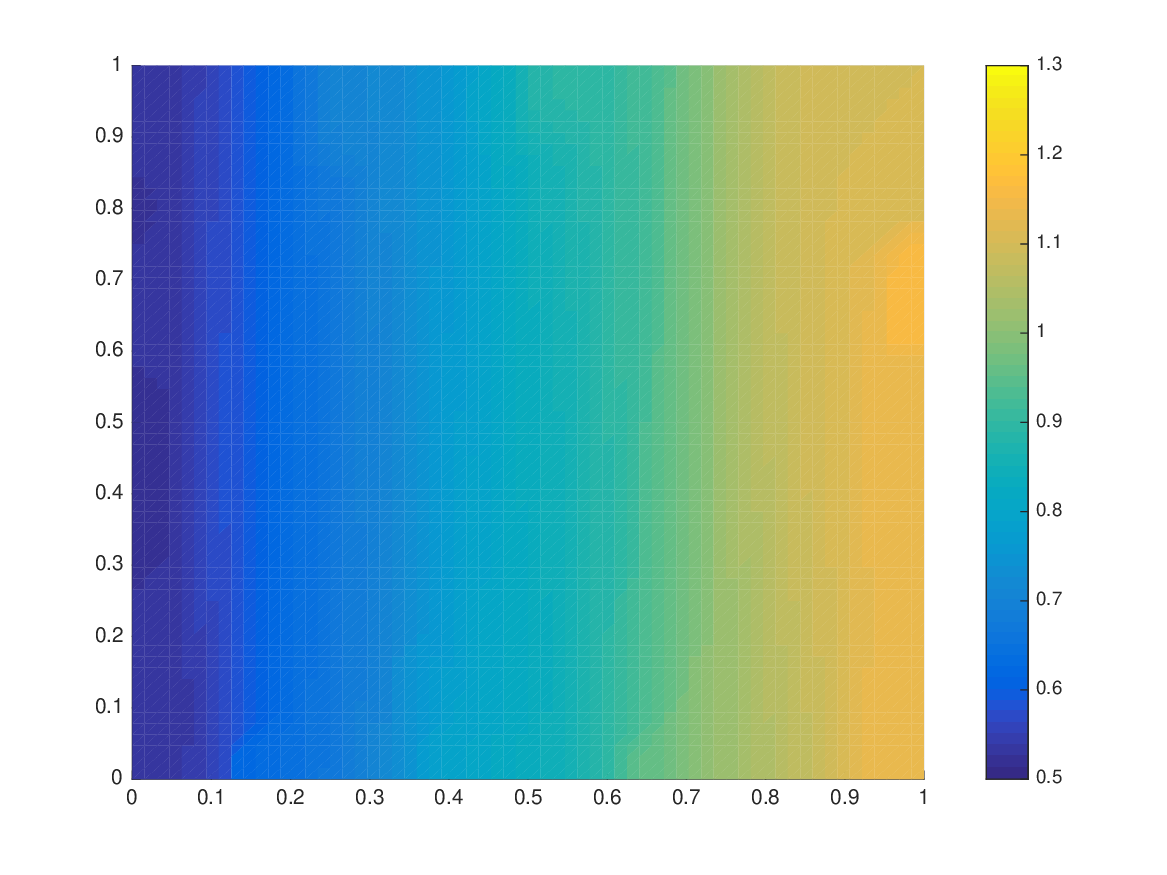}}
\caption{Plots of true conductivity and reconstruction with mesh size $h = 1/64$ in Example \ref{ex:2}.}
\label{example1}
\end{figure}

\begin{figure}[ht!]
\centering
\includegraphics[width=0.4\textwidth]{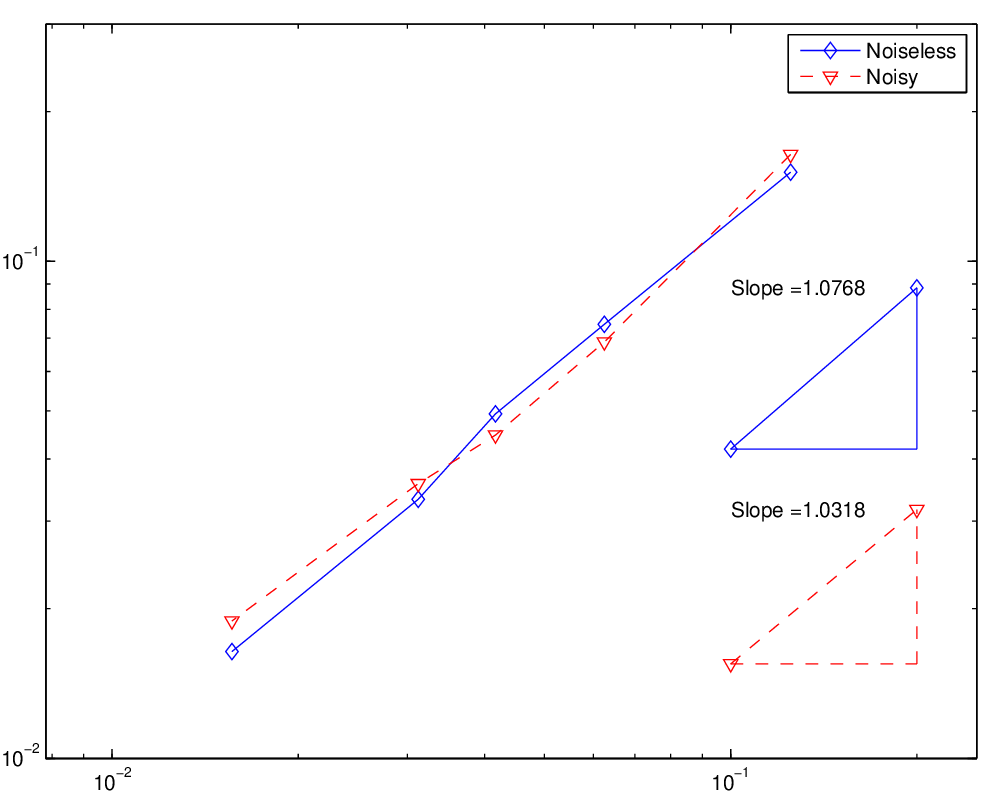}
\caption{The $L^2(\Omega)$ error $\|\sigma_h^\star - \sigma^\dag\|_{L^2(\Omega)}$ against the mesh size $h$ in Example \ref{ex:2}.}
\label{example1conv}
\end{figure}
%

\begin{exam}\label{ex:3}
The exact conductivity is  given by  
$\sigma^\dag=\sigma_0+0.2\exp(-8((x-0.6)^2+(y-0.6)^2)) $  in  $\Omega$ with the background conductivity $\sigma_0\equiv 1$. 
\end{exam}
In this example, the exact conductivity $\sigma^\dag$ consists of a smooth blob in a constant background as shown in Figure~\ref{Fig.sub.1}. The reconstruction $\sigma_h^\star$ from the noiseless measurement with the mesh size $h = 1/64$
are depicted in Figure~\ref{Fig.sub.2}, and the reconstruction from the  noisy measurement with noise level $\epsilon = 0.1\%$ is shown in Figure~\ref{Fig.sub.3}. The initial guess for the conductivity is given by $\sigma_0\equiv 1$ in $\Omega$.
 Despite some small oscillations near the boundary, it is observed that the  profile identifies well both the height and the location of the blob.  The reconstruction deteriorates only slightly in Figure~\ref{Fig.sub.3}, hence the proposed algorithm is stable with respect the the data noise. 

\begin{figure}[ht!]
\centering
\subfigure[True conductivity $\sigma^\dag$]{
\label{Fig.sub.1}
\includegraphics[width=0.32\textwidth]{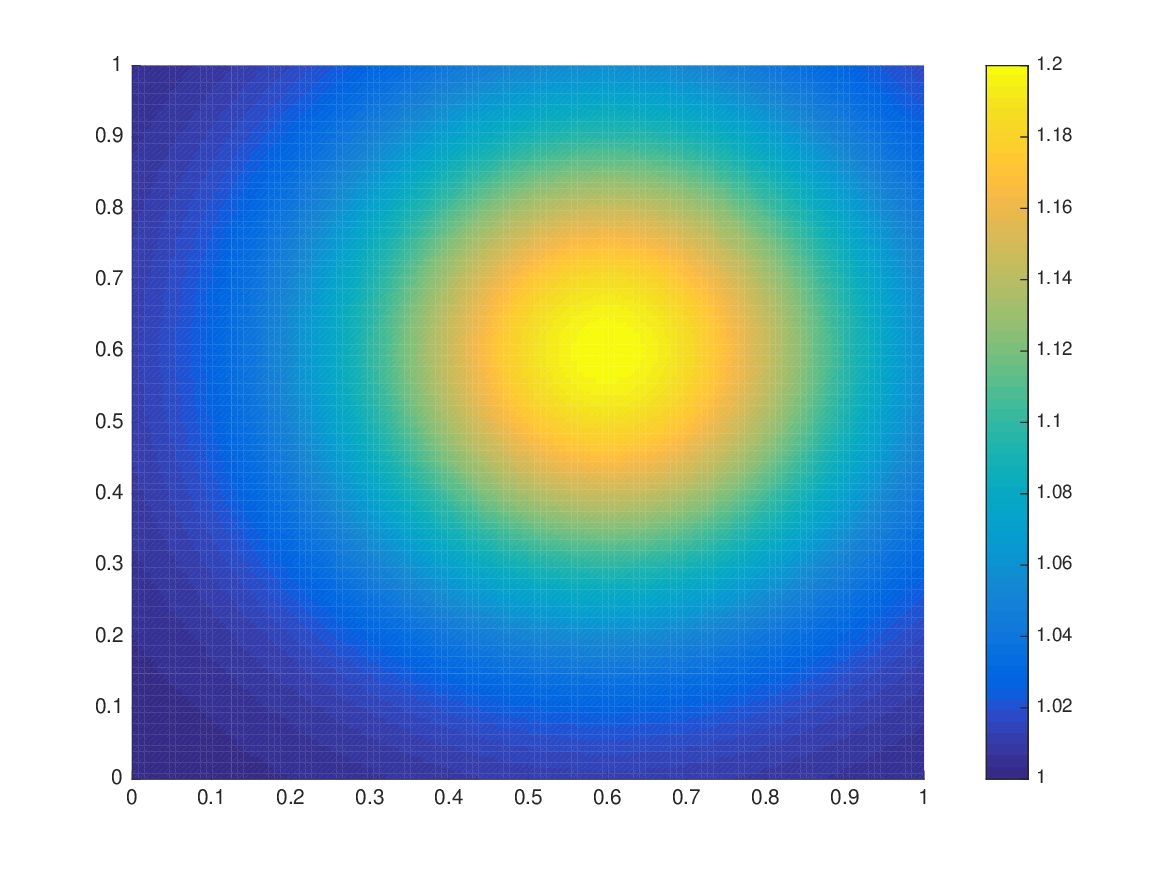}}
\subfigure[Reconstruction $\sigma_h^\star$ with noiseless measurements]{
\label{Fig.sub.2}
\includegraphics[width=0.32\textwidth]{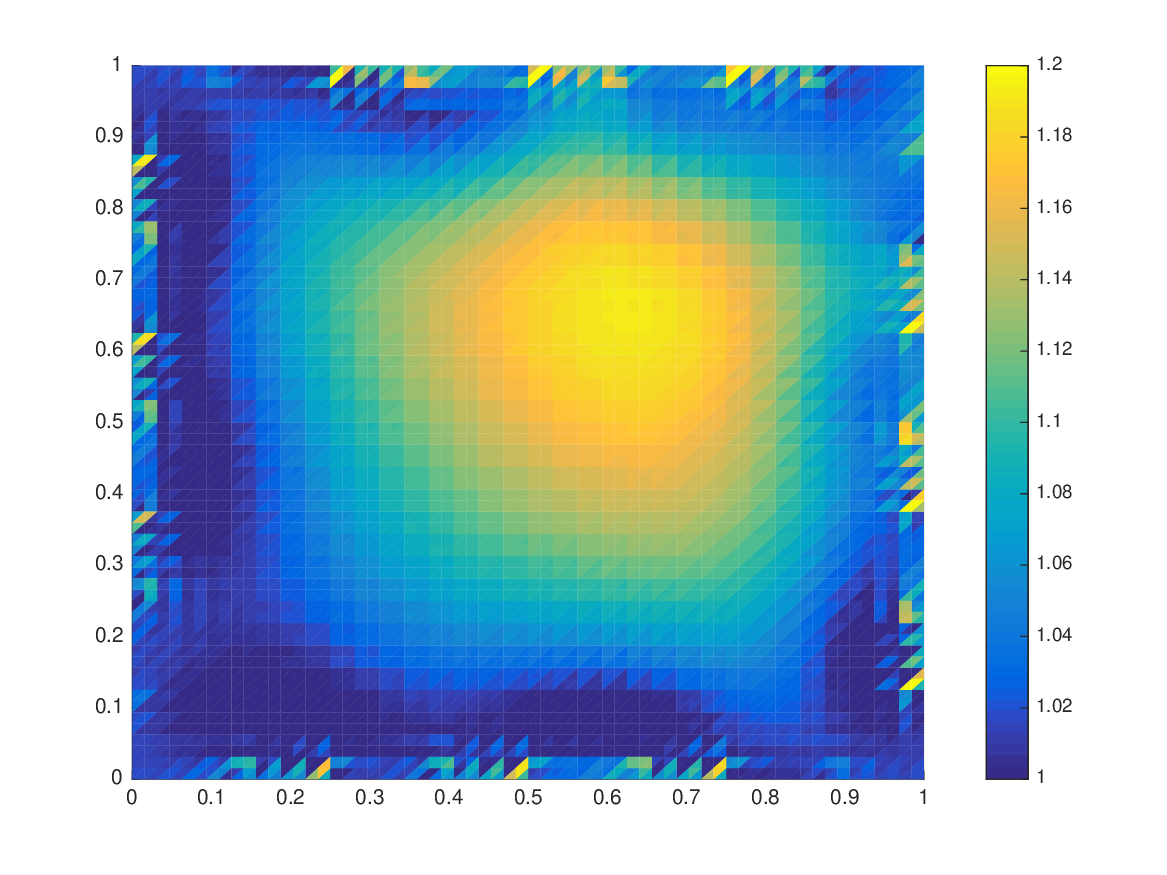}}
\subfigure[Reconstruction $\sigma_h^\star$ with noisy measurements]{
\label{Fig.sub.3}
\includegraphics[width=0.32\textwidth]{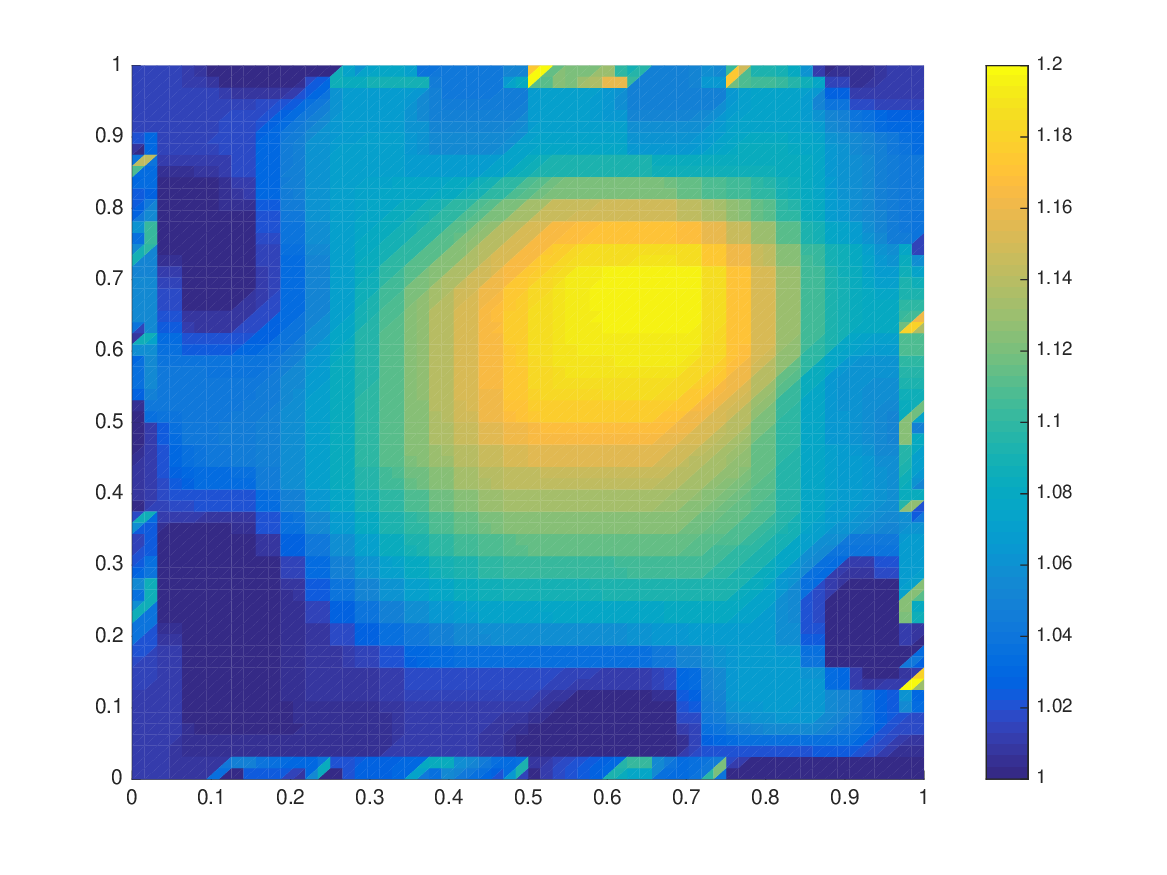}}
\caption{Plots of true  conductivity and reconstruction with mesh size $h = 1/64$ in Example \ref{ex:3}.}
\label{exmaple2}
\end{figure}


\begin{exam}\label{ex:4}
 We consider discontinuous conductivity fields in the following two scenarios:
 \begin{enumerate}
 \item The true conductivity field is given by $\sigma^\dag = \sigma_0+ 0.3\chi_{\Omega'}$, where $\chi_{\Omega'}$ is the characteristic function of the set $\Omega' = (0.1,0.3)\times(0.7,0.9)\cup (0.65,0.85)\times(0.1,0.3)$,  and the background conductivity $\sigma_0\equiv 1$.
\item The true conductivity field is given by $\sigma^\dag = \sigma_0+ 0.3\chi_{\Omega'}$, where $\chi_{\Omega'}$ is the characteristic function of the set $\Omega' = (0.15,0.35)\times(0.1,0.3)\cup (0.65,0.85)\times(0.1,0.3)\cup(0.15,0.35)\times(0.65,0.85)\cup (0.65,0.85)\times(0.65,0.85)$,  and the background conductivity $\sigma_0\equiv 1$.\end{enumerate}
\end{exam}

In our approach, the BV-based regularization method
allows us to seek general integrable functions with discontinuities.
We carry out experiments on such conductivity fields displayed in Figure~\ref{twoj.sub.1} and Figure~\ref{fourj.sub.1}, which feature with sharp high conductivity regions in the background. 
In these examples, we follow a refinement approach:
we start with  an initial guess for the conductivity $\sigma$ 
given by $\sigma_0\equiv 1$ in $\Omega$ with the mesh size $h = 1/32$, 
and solve the minimization problem \eqref{DM} using FISTA with 200 iterations.
Afterward, we refine the mesh size to $h = 1/64$, use the terminal solution from the coarser mesh with $h = 1/32$
as the initial guess,  and perform 80 FISTA iterations for the minimization problem \eqref{DM} on the refined triangulation. 
The true conductivities $\sigma^\dag$ and the reconstructions $\sigma_h^\star$ 
are depicted in Figure~\ref{Fig.lable5} and Figure~\ref{Fig.lable6}, and the noisy measurements are of noise level $\epsilon = 0.1\%$.

The observations from previous examples remain valid, as our numerical method successfully captures the support of the inhomogeneity and it is robust with respect to the measurement as shown in Figure~\ref{Fig.lable5} and Figure~\ref{Fig.lable6}. In Figure~\ref{Fig.lable5}, the two disjoint regions of inhomogeneity are captured and separated for both exact and noisy data.  As for a more challenging case in Figure~\ref{Fig.lable6}, the overall profile stands out clearly with the four supports of inhomogeneity identified and separated well. Although the magnitude of the reconstruction  suffers from a loss due to the BV penalty terms, the profile is still  reasonable in both cases with discontinuous exact conductivity, which verifies the effectiveness of the proposed algorithm.

\begin{figure}[ht!]
\centering
\subfigure[True conductivity $\sigma^\dag$]{
\label{twoj.sub.1}
\includegraphics[width=0.32\textwidth]{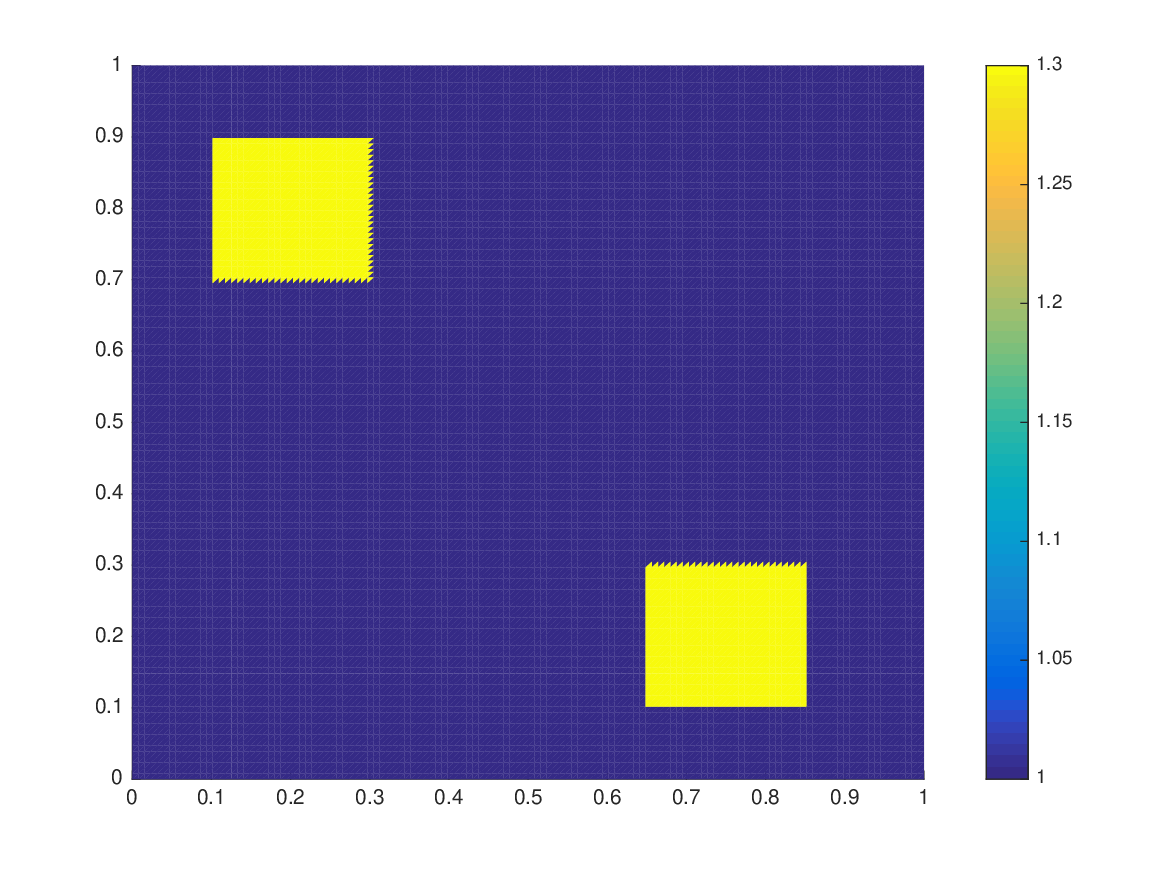}}
\subfigure[Reconstruction $\sigma_h^\star$ with noiseless measurements]{
\label{twoj.sub.2}
\includegraphics[width=0.32\textwidth]{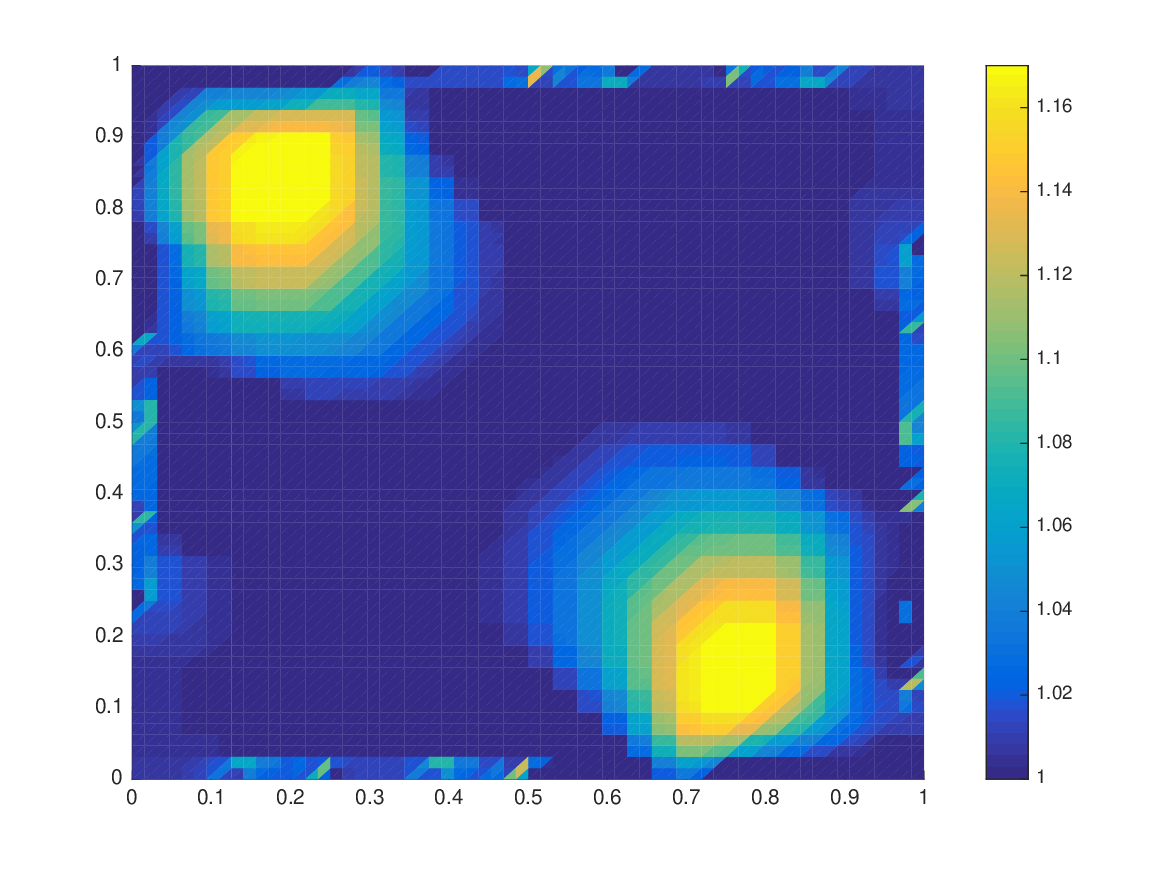}}
\subfigure[Reconstruction $\sigma_h^\star$ with noisy measurements]{
\label{twoj.sub.3}
\includegraphics[width=0.32\textwidth]{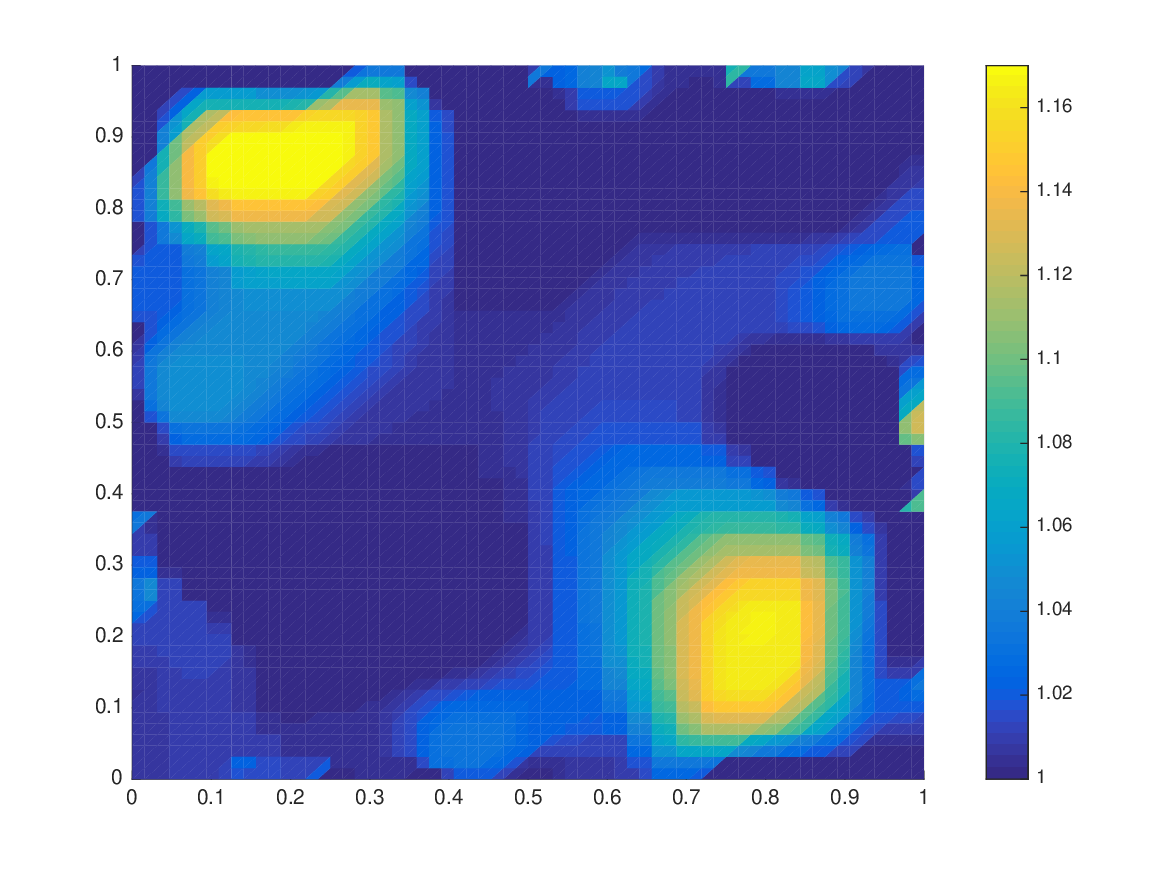}}
\caption{Plots of true conductivity and reconstruction in Example \ref{ex:4}, case 1.}
\label{Fig.lable5}
\end{figure}

\begin{figure}[ht!]
\centering
\subfigure[True conductivity $\sigma^\dag$]{
\label{fourj.sub.1}
\includegraphics[width=0.32\textwidth]{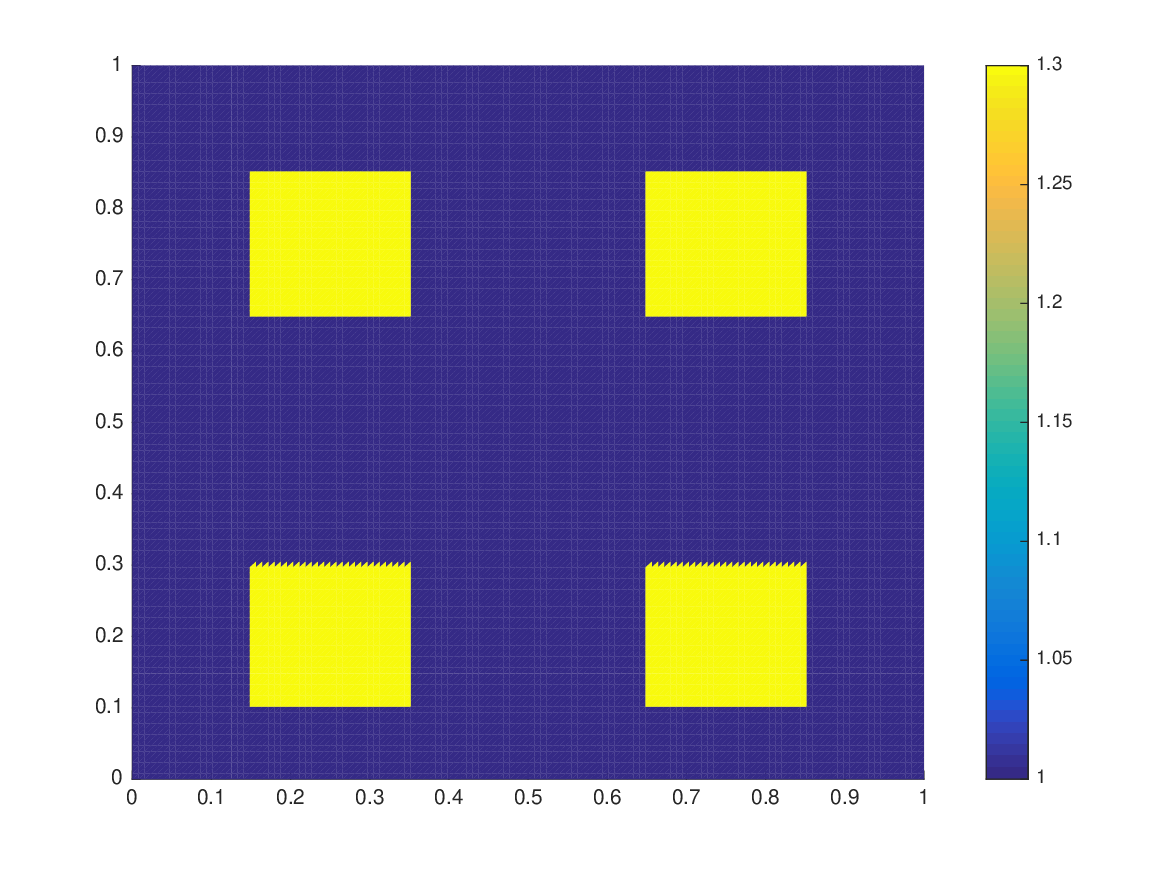}}
\subfigure[Reconstruction $\sigma_h^\star$ with noiseless measurements]{
\label{fourj.sub.2}
\includegraphics[width=0.32\textwidth]{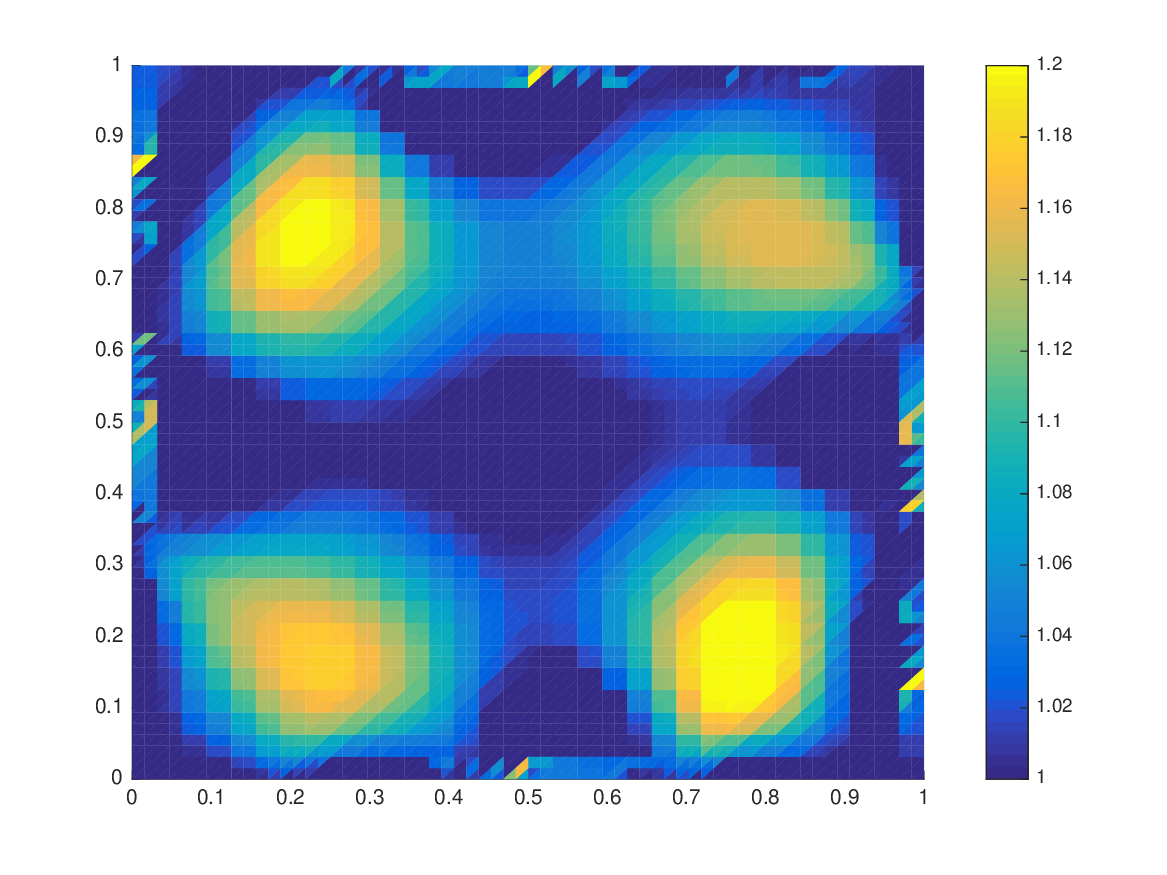}}
\subfigure[Reconstruction $\sigma_h^\star$ with noisy measurements]{
\label{fourj.sub.3}
\includegraphics[width=0.32\textwidth]{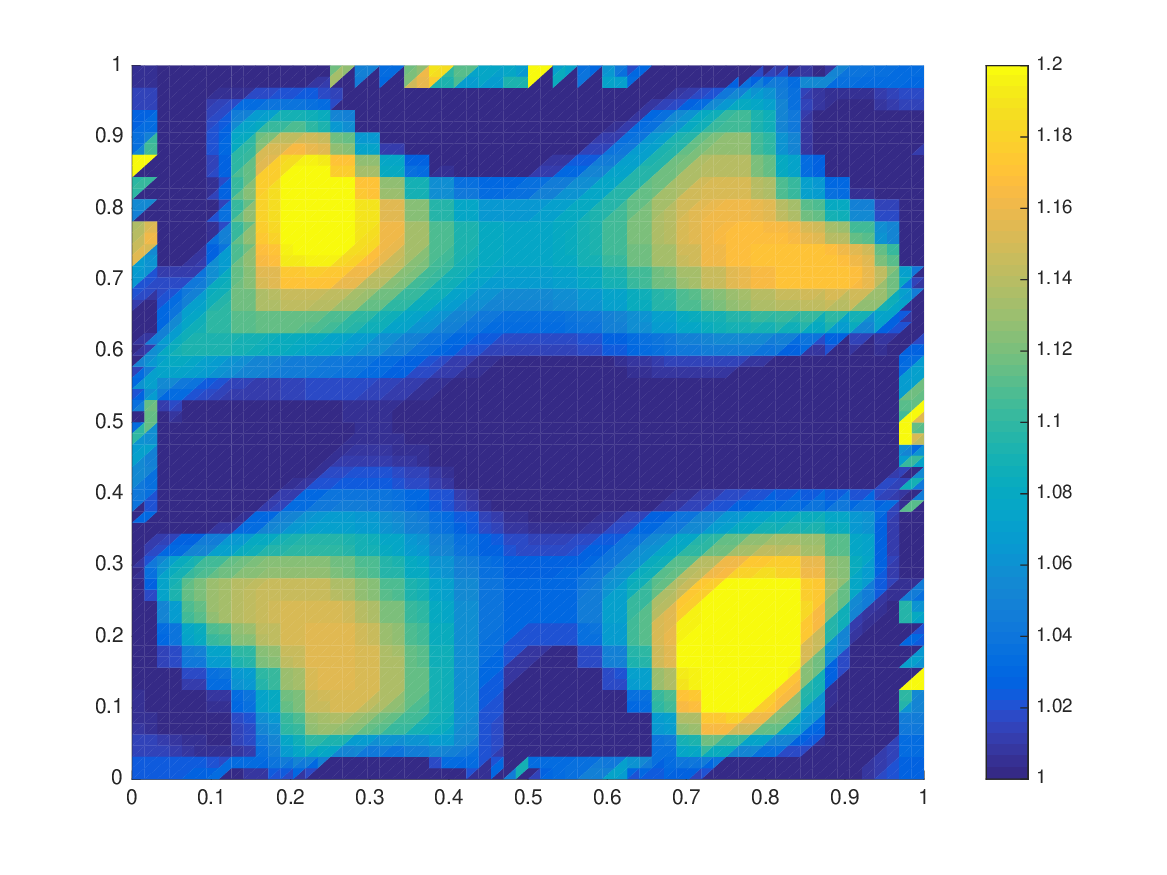}}
\caption{Plots of true conductivity and reconstruction in Example \ref{ex:4}, case 2.}
\label{Fig.lable6}
\end{figure}


\end{section}

\section{Concluding remarks}\label{conclusion}

In this paper, we present a numerical approach for solving the electrical impedance tomography problem. 
In the proposed approach, the forward problem is solved by the weak Galerkin method and the regularized minimization problem in the inverse process is numerically solved by FISTA. 
The error estimate is studied for the WG solution to the forward problem and  the convergence of the BV-based least-squares approach for the inverse process is established, in the sense that the sequence of discrete solutions contains a convergent subsequence to a solution of the continuous bounded variation regularization problem. Numerical experiments show that this approach provides convergent approximations for the non-smooth conductivity,  and it is  efficient and robust even for relative challenging cases without priori information on the shape of sought-for conductivity, which implies that it may have good potential applications in many real scenarios, e.g.,  geophysical imaging, cancer detection, and nondestructive testing. 



\appendix\section{Numerical algorithm for discrete inverse problem}\label{app1}

We will discuss the numerical algorithm for solving the 
discrete minimization problem with BV regularization \eqref{DM} on a family of regular mesh. 
We first introduce the coordinate representation $i_t: W_h \to \mathbb{R}^{N}$ of 
the space $W_h$ with respect to the standard basis $\{ \phi_j \}_{j=1}^{N}$ 
of piecewise constant functions on the mesh $\mathcal{T}_h$, 
where $N$ is the number of elements in the trangulation $\mathcal{T}_h$ and hence the dimension of $W_h$. 
For any $\sigma_h \in W_h$, the image of $\sigma_h$ under $i_t$ lies in the Euclidean space $\mathbb{R}^{N}$ 
consisting of the constant values of $\sigma_h$ restricted on the elements $T \in \mathcal{T}_h$. 
It directly follows that the image of $\mathcal{A}_h$ under $i_t$ is $\mathcal{C} = [\lambda, \lambda^{-1}]^{N}$. 
Then the discrete minimization problem \eqref{DM} can be reformulated as a problem in $\mathcal{C}$ given by 
\begin{equation}
\label{min_R}
\min_{\mathbf{x} \in \mathcal{C}} \{F(\mathbf{x}) = f(\mathbf{x}) + g(\mathbf{x})\},
\end{equation}
where
\begin{equation}
\begin{split}
f(\mathbf{x})&=\frac{1}{2}\Vert U(i_t^{-1}(\mathbf{x}))-U^\delta\Vert^2, \\
g(\mathbf{x})&=\alpha N_h((i_t^{-1}(\mathbf{x}))),
\end{split}
\end{equation}
 that is, $f$ corresponds to the measurement discrepancy which is a differentiable function, 
while $g$ corresponds to the total variation which is non-differentiable.

The problem \eqref{min_R} can be solved by the FISTA \cite{beck2009fast}  with the backtracking stepsize rule. 
To describe the method, we need to introduce a proximal map.
For any $L>0$, the map $p_L: \mathbb{R}^{N} \to \mathcal{C}$ is defined by 
\begin{equation}
\begin{split}
\label{prox}
p_L(\mathbf{y})
&= \text{prox}_{(1/L)g}\left(\mathbf{y}-\frac{1}{L}\nabla f(\mathbf{y})\right) 
=\mathop{\arg\min}_{\mathbf{x} \in \mathcal{C}} Q_L(\mathbf{x}, \mathbf{y}), 
\end{split}
\end{equation}
where $Q_L(\mathbf{x}, \mathbf{y})$ is the quadratic approximation of the functional $F$ at a given point $\mathbf{y}$:
\beqn\nb
Q_L(\mathbf{x}, \mathbf{y})=f(\mathbf{y})+\la \mathbf{x}-\mathbf{y},\nabla f(\mathbf{y})\ra
+\frac{L}{2}\Vert \mathbf{x}-\mathbf{y}\Vert^2+g(\mathbf{x}).
\eqn
The FISTA algorithm for solving \eqref{min_R} is summarized as follows:

{\centering
\begin{minipage}{0.9\linewidth}
\begin{algorithm}[H]
\caption{FISTA algorithm with the backtracking stepsize rule}\label{FISTA}
\begin{algorithmic}[1]
\State Input: search control parameter $\eta \in (0,1)$, maximum number of iterations $K$, tolerance $\delta$.
\State Initialization: step size $L_0>0$, initial guesses $\mathbf{y_1}=\mathbf{x}_0 \in \mathcal{C}$ and $t_1=1$.
\State Iteration: For $1 \leq k \leq K$, find the smallest integer $i_k \geq 0$ such that with $\hat{L}=\eta^{i_k}L_{k-1}$,
\beqn\nb
F(p_{\hat{L}}(\mathbf{y}_{k}))\leq Q_{\hat{L}}(p_{\hat{L}}(\mathbf{y}_{k}), \mathbf{y}_{k})\, .
\eqn
Set $L_k=\eta^{i_k}L_{k-1}$ and 
\beqnx
\mathbf{x}_k&=&p_{L_k}(\mathbf{y}_{k})\, ,\\
t_{k+1}&=&\dfrac{1+\sqrt{1+4t_k^2}}{2}\, ,\\
\mathbf{y}_{k+1}&=&\mathbf{x}_k+\frac{t_{k}-1}{t_{k+1}}(\mathbf{x}_k-\mathbf{x}_{k-1})\, .
\eqnx
If $\|\mathbf{y}_{k+1} - \mathbf{y}_k\| < \delta$, terninate. 
\end{algorithmic}
\end{algorithm}
\quad 
\end{minipage}
\par}

Next we introduce the numerical algorithm for the proximal map $p_L$ in \eqref{prox}, 
which involves the computation of the gradient of $f$ and 
the projection gradient method for the non-differentiable term $g$. 
To compute the $j$-th component of the gradient $\nabla f$ at $\mathbf{x} = i_t(\sigma_h)$, we observe that  
$$
\nabla f(\mathbf{x}) \cdot \mathbf{e}_j 
= \lim_{t \to 0} \dfrac{f(\mathbf{x}+t \mathbf{e}_j) - f(\mathbf{x})}{t}
= \lim_{t \to 0} \dfrac{J_{h,1}(\sigma_h+t \phi_j) - J_{h,1}(\sigma_h)}{t}
= J_{h,1}'(\sigma_h) \phi_j,
$$
where $J_{h,1}(\sigma_h) =\Vert U_h(\sigma_h)-U^\delta\Vert^2$ 
and $J_{h,1}'(\sigma_h) \alpha_h$ is the G\^{a}teaux derivative of $J_{h,1}$ 
at $\sigma_h$ in the direction $\alpha_h \in W_h$. 
To calculate the G\^{a}teaux derivative  $J_{h,1}'(\sigma_h) \alpha_h$, 
we  introduce an auxiliary dual problem:
find $(z_h, Z_h) \in \mathbb{H}_h$ such that 
\begin{equation}\label{aux}
a_s(\sigma_h, (z_h, Z_h), (v_h, V_h))=\la U_h(\sigma_h)-U^\delta, V_h\ra\, 
\quad  \forall  (v_h, V_h)\in \mathbb{H}_h\, .
\end{equation}
For $\alpha_h \in W_h$, if we denote $\varepsilon_h=u_h'(\sigma_h) \alpha_h$ and 
$\mathcal{E}_h=U_h'(\sigma_h)\alpha_h$, we obtain
$$J_{h,1}'(\sigma_h)\alpha_h=2\sum_{l=1}^L \int_{e_l}(U_h(\sigma_h)-U^\delta) \mathcal{E}_h ds.$$
By definition, $(\varepsilon_h,\mathcal{E}_h)$ satisfies for all $ (v_h, V_h)\in \mathbb{H}_h$,
\begin{equation}\label{gateaux}
(\alpha_h \nabla_w u_h, \nabla_w v_h) +(\sigma_h \nabla_w \varepsilon_h, \nabla_w v_h)+\sum_{l=1}^L z^{-1}_l \la\varepsilon_h-\mathcal{E}_h, v_h-V_h\ra_{e_l}+\sum_{T\in\mathcal{T}} h_l^{-1}\la Q_b\varepsilon_0 -\varepsilon_b, Q_b v_0- v_b\ra_{\partial T}=0.
\end{equation}
Taking $(v_h, V_h)=(\varepsilon_h,\mathcal{E}_h)$ in \eqref{aux} and $(v_h, V_h)=(z_h,Z_h)$ in \eqref{gateaux}, we can deduce
$$\sum_{l=1}^L \int_{e_l}(U_h(\sigma_h)-U^\delta) \mathcal{E}_h ds=-(\alpha_h \nabla_w u_h, \nabla_w z_h ).$$
Thus we have the formula of $j$-th component of $f$ 
\begin{equation}
\nabla f(\mathbf{x}) \cdot \mathbf{e}_j = -2(\phi_j \nabla_w u_h, \nabla_w z_h ). 
\end{equation}
Now the proximal map $p_L(\mathbf{y})$ can be reduced to a total variation-based denoising problem 
\begin{equation}
\mathop{\arg\min}_{\mathbf{x}\in\mathcal{C}} 
\left\{\Vert \mathbf{x}-\mathbf{d}\Vert^2+\alpha N_h(i_t^{-1}(\mathbf{x})) \right\},
\label{primalA}
\end{equation}
where $d =\mathbf{y}-\frac{1}{L}\nabla f(\mathbf{y}) $ and  $ N_h(i_t^{-1}(\mathbf{x}))$ is the $\ell_1$-based anisotropic total variation. 
We use the Fast Gradient Projection (FGP) \cite{beck2009fast} method to solve a dual problem of this denoising problem \eqref{primalA}, which is a
continuously differentiable convex minimization problem with a simple constraint set. Readers are referred to \cite{beck2009fast, beck2009fast1} for the relation between the primal and dual optimal solutions and the convergence rate of this algorithm.


\begin{thebibliography}{99}
\bibitem{Evans2010pde}
L. C. Evans, Partial Differential Equations, American Mathematical Society, Providence, R.I., 2010.

\bibitem{grisvard2011elliptic}
P. Grisvard, Elliptic Problems in Nonsmooth Domains, Society for Industrial and Applied Mathematics, 2011.


\bibitem{beck2009fast}A. Beck and M. Teboulle,
Fast gradient-based algorithms for constrained total variation image denoising and deblurring problems, IEEE Trans. Image Process, 18:2419--2434, 2009.

\bibitem{beck2009fast1}
A. Beck and  M. Teboulle,  A fast iterative shrinkage-thresholding algorithm for linear inverse problems. SIAM J. Imaging Sci., 2:183--202, 2009.


\bibitem{borsic2010vivo} A. Borsic,  B. M. Graham, A.  Adler,  and W. R. B. Lionheart,
In vivo impedance imaging with total variation regularization,  IEEE Trans. Medical Imaging, 29:44--54, 2010.



\bibitem{cheng1989electrode} K. S. Cheng, D Isaacson, J. C. Newell, and D. G. Gisser, Electrode models for electric current computed tomography,  IEEE Trans. Biomed. Engr., 36:918--924, 1989.
%
\bibitem{daneshmand20133d}P. G. Daneshmand and R. Jafari, A 3D hybrid BE--FE solution to the forward problem of electrical impedance tomography, Eng. Anal. Bound. Elem.,  37:757--764, 2013.



\bibitem{hakula2014reconstruction} H. Hakula, N. Hyv\"{o}nen, and M. Leinonen,
Reconstruction algorithm based on stochastic Galerkin finite element method for electrical impedance tomography, Inverse Problems, 30:1003--1029, 2014.

\bibitem{hallaji2014electrical}
M. Hallaji, A. Sepp\"{a}nen, and M. Pour-Ghaz,
Electrical impedance tomography-based sensing skin for quantitative imaging of damage in concrete, Smart Mater. Struct, 23:085001--085013, 2014.



\bibitem{jin2012reconstruction} B. Jin, T. Khan, and P. Maass, A reconstruction algorithm for electrical impedance tomography based on sparsity regularization, Internat. J. Numer. Methods Engrg., 89:337--353, 2012.


\bibitem{jin2012analysis}B. Jin and P. Maass, An analysis of electrical impedance tomography with applications to Tikhonov regularization, ESAIM Control Optim. Calc. Var., 18:1027--1048, 2012.

\bibitem{jin2016convergent}
B. Jin, Y. Xu, and J. Zou, A convergent adaptive finite element method for electrical impedance tomography,  IMA J. Numer. Anal., 37:1520--1550, 2017.




\bibitem{knudsen2009regularized} K. Knudsen, M. Lassas, J. L. Mueller, and S. Siltanen,  Regularized D-bar method for the inverse conductivity problem, Inverse Probl. Imaging, 35:99--624, 2009.


\bibitem{lechleiter2008factorization} A. Lechleiter,  N. Hyv\"{o}nen, and H. Hakula,  The factorization method applied to the complete electrode model of impedance tomography, SIAM J. Appl. Math.,  68:1097--1121, 2008.

\bibitem{lechleiter2006newton} A. Lechleiter and A. Rieder, Newton regularizations for impedance tomography: a numerical study, Inverse Problems, 22:1967--1987, 2006.

\bibitem{lin2015comparative}G. Lin, J. Liu, and F. Sadre-Marandi, A comparative study on the weak Galerkin, discontinuous Galerkin, and mixed finite element methods, J. Comput. Appl. Math., 273:346--362, 2015.

%

\bibitem{malone2014stroke} E. Malone, M. Jehl, S. Arridge, T. Betcke and D. Holder,
Stroke type differentiation using spectrally constrained multifrequency EIT: evaluation of feasibility in a realistic head model, Physiological Meas., 35:1051--1066, 2014.

\bibitem{moreau1965proximite} J. J. Moreau,
Proximit{\'e} et dualit{\'e} dans un espace hilbertien, Soc. Math. France, 93:273--299, 1965.

\bibitem{mu2014weak} L. Mu, J. Wang, and X. Ye, Weak Galerkin finite element methods for the biharmonic equation on polytopal meshes, Numer. Methods Partial Differential Equations, 30:1003--1029, 2014.


\bibitem{mu2015weak}L. Mu, J. Wang, and X. Ye,
A weak Galerkin finite element method with polynomial reduction, J. Comput. Appl. Math., 285:45--58, 2015.

\bibitem{murai1985electrical} T. Murai and Y. Kagawa, Electrical impedance computed tomography based on a finite element model, IEEE Trans. Biomed. Eng., BME-32:177--184, 1985.

\bibitem{rudin1992nonlinear}
L. I. Rudin, S. Osher and E. Fatemi, Nonlinear total variation based noise removal algorithms,  Phys. D, 60:259--268, 1992.

\bibitem{somersalo1992existence} E. Somersalo,  M. Cheney, and D. Isaacson, Existence and uniqueness for electrode models for electric current computed tomography, SIAM J. Appl. Math., 52:1023--1040, 1992.

\bibitem{scherzer2011handbook} O. Scherzer, Handbook of Mathematical Methods in Imaging, Springer Science Business Media, 2010.

%
%
%

\bibitem{wang2013weak}
J. Wang and X. Ye, A weak Galerkin finite element method for second-order elliptic problems, J. Comput. Appl. Math., 241:103--115, 2013.

\bibitem{wang2014weak}
J Wang and X Ye,  A weak Galerkin mixed finite element method for second order elliptic problems, Mathematics of Computation, 83 (2014), pp. 2101-2126.

\bibitem{wang2016weak} J. Wang and X. Ye, A weak Galerkin finite element method for the Stokes equations,  Adv. Comput. Math., 42:155--174, 2016.

\bibitem{wang2015new} M. Wang, J. Jia, Y. Faraj, Q. Wang, C. Xie, G. Oddie, K. Primrose, and C. H. Qiu, A new visualisation and measurement technology for water continuous multiphase flows, Flow Meas. Instrum., 46:204--212, 2015.

\bibitem{winkler2014resolution} R. Winkler and A. Rieder, Resolution-controlled conductivity discretization in electrical impedance tomography,  SIAM J. Imaging Sci., 7:2048--2077, 2014.


\bibitem{yousefi2013combined} M. R. Yousefi, R. Jafari, and H. A. Moghaddam, A combined wavelet-based mesh-free method for solving the forward problem in electrical impedance tomography, IEEE Trans. Instrum. Meas., 62:2629--2638, 2013.


\bibitem{yu2000modified}
M. Yu and D. E. Dougherty, Modified total variation methods for three-dimensional electrical resistance tomography inverse problems, Water Resour.  Res., 36:1653--1664, 2000.

\bibitem{zeng2017convergence}
Y. Zeng, J. Chen, and F. Wang, Convergence analysis of a modified weak Galerkin finite element method for Signorini and obstacle problems, Numer. Methods for Partial Differential Equations, 33(5):1459--1474, 2017.

\bibitem{asas1999regularization}
E. Casas, K. Kunisch, and C. Pola, Regularization by functions of bounded variation and applications to image enhancement, Appl. Math. Optim., 40(2):229--257, 1999.


\bibitem{brezzi1991mixed}
F. Brezzi and M. Fortin, Mixed and Hybrid Finite Element Methods, Springer, New York, 1991.
\bibitem{Bastian2003dg}
P. Bastian and B. Rivi\'ere, Superconvergence and H(div) projection for discontinuous Galerkin methods, Internat. J. Numer. Methods Fluids, 42:1043--1057, 2003.

 \bibitem{Lukaschewitsch2003Tik}
 M. Lukaschewitsch, P. Maass, and M. Pidcock,  Tikhonov regularization for electrical impedance tomography on unbounded domains, Inverse Problems, 19(3):585--610, 2003.
 
 \bibitem{Chung2005tv}
 E. T. Chung, T. F. Chan, and X. C. Tai, Electrical impedance tomography using level set representation and total variational regularization, J. Comput. Phys., 205(1):357--372, 2005.

\bibitem{Somersalo1992eit}
E. Somersalo, M. Cheney, and D. Isaacson, Existence and uniqueness for electrode models for electric current computed tomography, SIAM J. Appl. Math., 52(4):1023--1040, 1992.

\bibitem{adesokan2019acousto}
 B. J. Adesokan, B. Jensen, B. Jin, and K. Knudsen, Acousto-electric tomography with total variation regularization, Inverse Problems, 35(3):035008, 2019.
 
 \bibitem{giusti1984minimal}
 E. Giusti and G. H. Williams, Minimal Surfaces and Functions of Bounded Variation, Springer, Boston, 1984.
 \bibitem{lin2014weak}
 G. Lin, J. Liu, L. Mu, and X. Ye,  Weak Galerkin finite element methods for Darcy flow: Anisotropy and heterogeneity, J. Comput. Phys., 276:422--437, 2014.
 
\end{thebibliography}
\end{document}